\documentclass[11pt]{amsart}

\usepackage{amsmath}
\usepackage{amssymb}
\usepackage{graphicx}

\usepackage{amsmath, amsthm, amsfonts}
\usepackage{amssymb}
\usepackage{pdfsync}
\usepackage{parskip}
\usepackage{hyperref}
\usepackage{mathrsfs} 
\usepackage{mathtools}
\usepackage{mdframed}
\usepackage{tikz}

\numberwithin{equation}{section}
\newtheorem{theorem}{Theorem}[section]
\newtheorem{prop}[theorem]{Proposition}
\newtheorem{definition}[theorem]{Definition}
\newtheorem{lem}[theorem]{Lemma}

\theoremstyle{remark}
\newtheorem{remark}[theorem]{Remark}

\newcommand{\cA}{\mathcal A}

\newcommand{\cS}{\mathcal S}
\newcommand{\cT}{\mathcal T}

\def\M{\mathsf{M}}

\def\p{\mathsf{p}}
\def\q{\mathsf{q}}

\def\R{\mathbb{R}}
\def\C{\mathbb{C}}
\def\N{\mathbb{N}}
\def\Nz{\mathbb{N}_0}

\def\bE{\mathbb{E}}
\def\M{\mathsf{M}}
\def\p{\mathsf{p}}
\def\q{\mathsf{q}}

\def\cS{\mathcal{S}}
\def\re{{\rm Re}}

\def\L{\mathcal{L}}
\def\Lis{\mathcal{L}{\rm{is}}}
\def\F{\mathfrak{F}}

\def\B{\mathbb{B}}

\def\cA{\mathcal{A}}

\def\a{\mathfrak{a}}

\def\sg{{\rm{sign}}}

\def\S{\mathcal{S}}

\def\sL{\mathscr{L}}

\def\max{{\rm max}}
\def\dom{{\rm dom}}

\def\fb{\mathsf{b}}

\title[Maximal $L_q$-regularity of nonlocal parabolic equations]{Maximal $L_q$-regularity of nonlocal parabolic equations in higher order Bessel potential spaces}

\author[N. Roidos]{Nikolaos Roidos}
\address{Department of Mathematics, University of Patras, 26504 Rio Patras, Greece}
\email{roidos@math.upatras.gr}

\author[Y. Shao]{Yuanzhen Shao}
\address{Department of Mathematics, The University of Alabama, Box 870350, Tuscal\-oosa, AL 35487-0350, USA}
\email{yshao8@ua.edu}



\begin{document}

\begin{abstract}
We consider fractional parabolic equations with variable coefficients and establish maximal $L_{q}$-regularity in Bessel potential spaces of arbitrary nonnegative order. As an application, we show higher order regularity and instantaneous smoothing for the fractional porous medium equation and for a nonlocal Kirchhoff equation. 
\end{abstract}

\maketitle


\section{Introduction}

The mathematical study of diffusion has been very successful in modelling and analyzing a variety of phenomena in physics, chemistry, biology, material sciences, population dynamics and finance. However, standard diffusion models are incapable of describing long-memory or long-range interactions in real processes and substances and thus fail to explain a number of phenomena occurring in recent experiments, cf. \cite{Cardozo, Carreras, Carreras2, Gentle, Goldstone}. This observation explains the recent surge of the study of nonlocal differential equations. In this article, we will focus on a class of nonlocal parabolic equations.
 
It is well known that maximal $L_q$-regularity theory has been playing an important role in the study of nonlinear parabolic equations. In particular, maximal regularity theory is capable of handling systems (and thus tensor-valued equations) of parabolic equations and higher order equations, which are generally not accessible in many other traditional approaches, like monotone operator techniques, a priori estimates and Leray-Schauder continuation techniques. 

The main theme of the paper is to investigate the maximal $L_q$-regularity of the following nonlocal parabolic equation
\begin{equation}
\label{S1: MREQ 1}
\left\{\begin{aligned}
\partial_t u +w \sL^\sigma u &=f &&\text{on}&&\M\times (0,\infty);\\
u(0)&=u_0 &&\text{on}&&\M.
\end{aligned}\right.
\end{equation}
Here $(\M,g)$ is either an $n$-dimensional smooth closed Riemannian manifold or the Euclidean space $(\R^n, g_n)$ with $g_n$ being the standard Euclidean metric in $\R^n$.
In this article, a closed manifold always refers to one that is compact and without boundary. The coefficient $w$ belongs to $BC^r(\M)$ with $r\in (0,\infty)$.
Briefly speaking, a function belongs to $BC^k(\M)$ if its derivatives up to order $k$ are continuous and bounded.
See Section~\ref{Section 2.1} for the precise definitions of function spaces.
For any $u \in C^\infty(\M,T \M^{\otimes \eta }\otimes T^* \M^{\otimes \tau })$, i.e. smooth tensor-valued functions, the linear differential operator $\sL$ is defined by
\begin{equation}
\label{def of sL}
\sL u = \nabla^* (\a \nabla u) \quad \text{with } \a \in BC^\infty(\M),
\end{equation}
where $\nabla^*$ is the formal adjoint of 
$$
\nabla: C^\infty(\M,T \M^{\otimes \eta }\otimes T^* \M^{\otimes \tau })\to C^\infty(\M,T \M^{\otimes \sigma }\otimes T^* \M^{\otimes (\tau+1) } ).
$$
In addition, we assume that there exists a constant $c>0$ such that
$$
c^{-1} \leq \a \leq c.
$$
Lower regularity of $\a$ or tensor-valued $\a$ can be imposed. For the sake of simplicity, we will confine our discussion to the case $\a \in BC^\infty(\M)$ in this manuscript.
Further, $\sigma \in (0,1)$ and $\sL^\sigma$ is understood as the fractional power of $\sL$.

Given a continuously and densely embedded Banach couple $X_1\overset{d}{\hookrightarrow}X_0$, assume that the linear operator $-\cA$, with $\dom(\cA)=X_1$, generates a strongly continuous analytic semigroup on $X_0$. For any $q\in (1,\infty)$, the following abstract Cauchy problem 
\begin{equation}
\label{S1: Cauchy problem}
\left\{\begin{aligned}
\partial_t u(t) +\cA u(t) &=f(t), &&t\geq 0\\
u(0)&=u_0 &&
\end{aligned}\right. 
\end{equation}
is said to have maximal $L_q$-regularity if for any 
$$
f\in L_q(\R_+, X_0) \quad \text{and} \quad u_0\in X_{1/q,q}:= (X_0,X_1)_{1-1/q,q},
$$
\eqref{S1: Cauchy problem} has a unique solution
$$
u\in L_q(\R_+, X_1) \cap H^1_q(\R_+, X_0) . 
$$
Here $(\cdot,\cdot)_{\theta,q}$ with $\theta\in (0,1)$ is the real interpolation method, cf. \cite[Example I.2.4.1]{Ama95}, and $H_{q}^{1}$ is the usual Bessel potential space.
Symbolically, we denote the maximal $L_q$-regularity property by
\begin{equation*}
\cA\in \mathcal{MR}_q(X_1,X_0).
\end{equation*}
In combination with an abstract Theorem by P. Cl\'ement and S. Li, cf. Theorem~\ref{ClementLi}, maximal $L_q$-regularity of \eqref{S1: MREQ 1} can be used to establish the local well-posedness of a large class of quasilinear parabolic equations including
\begin{equation}
\label{S1: Cauchy problem-quaslinear}
\left\{\begin{aligned}
\partial_t u(t) + w(u(t)) \sL^\sigma u(t) &= F(t,u(t)) &&\text{on}&&\M\times (0,\infty);\\
u(0)&=u_0 , &&\text{on}&&\M,
\end{aligned}\right. 
\end{equation}
where 
$$
w\in C^{1-}(U,BC^r(\M)) \text{ and } F\in C^{1-,1-}([0,T_0]\times U, H^s_p(\M, T \M^{\otimes \sigma }\otimes T^* \M^{\otimes \tau }))
$$ for some $U\subseteq X_{1/q,q}$ open and $T_0>0$, and 
$$
u_0\in B^{s+ 2\sigma -2\sigma/q }_{p,q}(\M, T \M^{\otimes \sigma }\otimes T^* \M^{\otimes \tau } ).
$$
Here $r>s\geq 0$ and $H^s_p(\M, T \M^{\otimes \sigma }\otimes T^* \M^{\otimes \tau }), B^{s+ 2\sigma -2\sigma/q }_{p,q}(\M, T \M^{\otimes \sigma }\otimes T^* \M^{\otimes \tau } )$ are tensor-valued Bessel potential and Besov spaces, respectively. See Section~\ref{Section 2.1} for details.

Particularly, \eqref{S1: Cauchy problem-quaslinear} includes the fractional porous medium equation of the form
\begin{equation*}
\left\{\begin{aligned}
\partial_t u +(-\Delta_g )^\sigma (|u|^{m-1}u )&=0 &&\text{on}&&\M\times (0,\infty);\\
u(0)&=u_0 &&\text{on}&&\M,
\end{aligned}\right.
\end{equation*}
which has been extensively investigated in the last decade. See \cite{BonFigOto17, BonFigVaz18, BonSireVaz15, BonVaz15, BonVaz16, PalRodVaz11, PalRodVaz12, VPGR, Vaz14, Vaz15} for instance. For an approach to the above problem in the framework of pseudodifferential operators, we also refer to the theory developed in \cite{Grubb1, Grubb2, Grubb3}.

In the realm of nonlocal parabolic problems, maximal regularity theory has been successfully applied to several models, cf. \cite{Ama05, Ama07, Gui09, GL07, G131, GuiShao17}.
However, when the leading term is nonlocal in space and has a variable multiplier, e.g. $w \sL^\sigma$ in \eqref{S1: MREQ 1}, an essential difficulty arises. 
Indeed, in the conventional approaches, the first step of obtaining the maximal regularity property of a parabolic problem is to study a constant-coefficient problem by means of proper harmonic analysis techniques. Then the perturbation theory of $R$-sectorial operators, cf. Definition~\ref{rsec}, and a freezing-of-coefficient method can be used to extend the maximal regularity property to equations/systems with variable coefficients.
When the leading nonlocal operator is accompanied by a variable coefficient, due to the spatial nonlocality, the standard freezing-of- coeffients method no longer applies. 

In recent work \cite{RoidosShao}, we have overcome the aforementioned difficulty and established the $L_q$-maximal regularity of \eqref{S1: MREQ 1} for the scalar case and $X_0=L_q(\M)$.
The main contribution of this article is to extend our previous results to tensor-valued equations and $X_0=H^s_p(\M)$ for arbitrary $s\geq 0$, i.e. Bessel potential spaces of arbitrary non-negative order.
As a direct consequence, immediate regularization of the corresponding nonlocal diffusion can be obtained. 
We expect that the methods in this and our previous work \cite{RoidosShao} will serve as the step stone to the study of more general nonlocal parabolic equations and systems.
 
This manuscript is organized as follows. In Section~\ref{Section 2.1}, we give the precise definitions of tensor-valued function spaces and present various properties of those spaces. 
In Section~\ref{Section 2.2}, we introduce some crucial functional analytic concepts in the study of maximal $L_q$-regularity like sectorial operators, bounded imaginary powers and $R$-sectorial operators.
In Section~\ref{Section 3}, we prove that the $H^s_p$-realization of the operator $c+\sL$ has bounded imaginary powers for sufficiently large constant $c>0$.
In Section~\ref{Section 4}, we obtain the $R$-sectoriality of the fractional operator $\sL^\sigma$. 
Section~\ref{Section 5} is devoted to the proof of the maximal $L_q$-regularity of \eqref{S1: MREQ 1}.
The key component of the proof consists of several novel commutator estimates.
In Section~\ref{Section 6}, we apply the maximal $L_q$-regularity result to two nonlocal quasilinear parabolic equations.
Lastly, the arguments for closed manifolds and Euclidean spaces are essentially the same. The only differences appear in Step 2 of the proof of Theorem~\ref{Thm: sectoriality Lap} and the proof of Lemma~\ref{Lem: localization est}.
Therefore, in this article, we will mainly focus on the case of closed manifolds and point out changes for $(\R^n, g_n)$ where necessary.

 \textbf{Notations:} 
For any two Banach spaces $X$ and $Y$, 
$$
X\doteq Y
$$ 
means that they are equal in the sense of equivalent norms. The notations 
$$
X\hookrightarrow Y, \quad X\xhookrightarrow{d} Y
$$ 
mean that $X$ is continuously embedded and further densely embedded into $Y$, respectively. $\L(X,Y)$ denotes the set of all bounded linear maps from $X$ to $Y$, and 
$$
\L(X):=\L(X,X). 
$$
Furthermore, $\Lis(X,Y)$ stands for the subset of $\L(X,Y)$ consisting of all bounded linear isomorphisms from $X$ to $Y$. Given a densely-defined operator $\cA$ in $X$, $\dom(\cA)$ stands for the domain of $\cA$. In addition, $\N$ denotes the set of nonzero natural numbers and $\Nz:=\N\cup \{0\}$.

\section{Preliminaries}\label{Section 2}

\subsection{Function spaces}\label{Section 2.1}

Given $\eta,\tau\in\N_0$, we define the $(\eta,\tau)$--tensor bundle of $\M$ as
$$
T^\eta_\tau \M:=T \M^{\otimes \eta }\otimes T^* \M^{\otimes \tau }, 
$$ 
where $T \M $ and $T^* \M $ are the tangent and the cotangent bundle of $\M$, respectively.
Let $\mathcal{T}^\eta_\tau \M$ denote the $C^\infty ( \M )$--module of all smooth 
sections of $T^\eta_\tau \M$.

Throughout the rest of this paper, we will adopt the following conventions.
\smallskip
\begin{mdframed}
\begin{itemize}
\item $(\M ,g)$ is either an $n$-dimensional closed manifold or $(\R^n,g_n)$.
\item $\p$ always denotes a generic point on $\M$.
\item $1<p,q<\infty$, $k\in \N_0$ and $r,s\geq 0$.
\item $\eta,\tau\in \N_0$, $V=V^{\eta}_{\tau}:=\{T^{\eta}_{\tau}\M, (\cdot|\cdot)_g\}$.
\item $\nabla$ is the extension of the Levi-Civita connection over $\cT^{\eta}_{\tau}{\M}$.
\item $|a|_g:=\sqrt{(a|\overline{a})_g}$ for all $a\in V$ is the (vector bundle) norm induced by $g$.
\end{itemize}
\end{mdframed}

The Sobolev space $W^{k}_p({\M},V)$ is defined as the completion of $C_{0}^{\infty}(\M,V)$, the space of smooth and compactly supported tensor-valued functions, in $L_{1,loc}(\M,V)$ with respect to the norm
\begin{align*}
\|\cdot\|_{k,p}: u\mapsto(\sum_{i=0}^{k}\||\nabla^{i}u|_{g}\|_p^p)^{\frac{1}{p}}.
\end{align*}
It is clear that $W^0_p(\M,V)\doteq L_p(\M,V)$. The Bessel potential spaces are defined by means of interpolation
\begin{align}\label{defofHspace}
H^s_p(\M,V):=
\begin{cases}
[W^k_p({\M},V),W^{k+1}_p(\M,V)]_{s-k} \quad &\text{ for }k<s<k+1,\\
[W^{k-1}_p(\M,V),W^{k+1}_p(\M,V)]_{1/2} &\text{ for }s=k\in\N,\\
L_p(\M,V) &\text{ for }s=0.
\end{cases}
\end{align}
Here $[\cdot,\cdot]_{\theta}$ is the complex interpolation method \cite[Example~I.2.4.2]{Ama95}. In particular, by \cite[Corollary~7.2~(i)]{Ama13}, $H^{k}_p({\M},V)\doteq W^{k}_p({\M},V)$. We denote the norm of $H^s_p(\M,V)$ by $\|\cdot\|_{s,p}$.

The following interpolation theory for Bessel potential spaces is proved in \cite[Corollary~7.2~(ii)]{Ama13}, \cite[Theorem~2.4.2]{TriebelbookI} and \cite[Theorem~7.4.4]{TriebelbookII}. 
\begin{prop}\label{Prop: interpolation}
Suppose that $0\leq s_0<s_1<\infty$ and $\theta\in (0,1)$. Then
$$
H^{s_\theta}_p(\M,V) \doteq [H^{s_0}_p(\M,V),H^{s_1}_p(\M,V)]_\theta
$$
and
$$
B^{s_\theta}_{p,q}(\M,V) \doteq (H^{s_0}_p(\M,V),H^{s_1}_p(\M,V))_{\theta,q},
$$
where $s_\theta=(1-\theta)s_0 + \theta s_1$.
\end{prop}
Here, $B^{s}_{p,q}(\M,V) $ is a Besov space, cf. \cite{TriebelbookI, TriebelbookII}. We define 
$$
BC^k(\M,V):=(\{u\in{C^k(\M,V)}:\|u\|_{k,\infty}<\infty\},\|\cdot\|_{k,\infty} ),
$$
where $\|u\|_{k,\infty} :=\max_{0\leq i \leq k}\||\nabla^i u|_g\|_{\infty}$.
Set
$$
BC^\infty(\M,V):=\bigcap_k BC^k(\M,V)
$$
endowed with the conventional projective topology. Then
\begin{center}
$bc^k(\M,V):=$ the closure of $BC^{\infty}(\M,V)$ in $BC^k (\M,V)$.
\end{center}
Letting $k<s<k+1$, the H\"older space $BC^s(\M,V)$ is defined by
$$
BC^s(\M,V):=(bc^k(\M,V),bc^{k+1}(\M,V))_{s-k,\infty}.
$$
Here $(\cdot,\cdot)_{\theta,\infty}$ is again the real interpolation method.
 
When $s\in (0,1)$, a function $u\in BC^s(\M,V)$ iff $u\in BC(\M,V)=BC^0(\M,V)$ and
$$
\|u\|_{s,\infty}= \|u\|_\infty + \sup\limits_{\p,\q\in \M} \frac{|u(\p)-u(\q)|_g}{d(\p,\q)^s} < \infty,
$$
where $d=d(\p,\q)$ the geodesic distance between two points $\p,\q\in\M$ with respect to the metric $g$.
This alternative characterization is well-known for Euclidean spaces. When $(\M,g)$ is a closed manifold, it can be proved via localization.
 
\begin{prop}\label{embBCH}
{\em (i)} For any $s_2>s_1>s_0\geq 0$, we have
\begin{equation}
\label{Sobolev embedding 0}
H_{p}^{s_2}(\M,V) \hookrightarrow B^{s_1}_{p,q}(\M,V) \hookrightarrow H_{p}^{s_0}(\M,V).
\end{equation}
Suppose that $s>r+n/p$ and $r\geq 0$. Then 
\begin{equation}
\label{Sobolev embedding}
H_{p}^{s}(\M,V) \hookrightarrow BC^{r}(\M,V) \quad \text{and}\quad B_{p,q}^{s}(\M,V) \hookrightarrow BC^{r}(\M,V).
\end{equation}
{\em (ii)} Let $v\in H^{\nu}_p(\M,V)$. When $s\geq \nu $, given any $w\in BC^s(\M)$ we have
\begin{equation}
\label{pointwise mul 0}
\|wv\|_{H^{\nu}_p(\M,V)}\leq C\|w\|_{BC^s(\M)}\|v\|_{H^{\nu}_p(\M,V)},
\end{equation}
for certain $C>0$. Further, if $w\in H^{\xi+\frac{n}{q}}_{q}(\M)$ for some $\xi>\nu\geq0$ and $q\in(1,\infty)$, then
\begin{equation}
\label{pointwise mul 1}
\|wv\|_{H^{\nu}_p(\M,V)}\leq C_{0}\|w\|_{H^{\xi+\frac{n}{q}}_{q}(\M)}\|v\|_{H^{\nu}_p(\M,V)},
\end{equation}
for certain $C_{0}>0$. \\
{\em (iii)} Let $(\M,g)$ be an $n$-dimensional smooth closed Riemannian manifold. When $s> n/p$, $\F^s(\M)$ is a Banach algebra (up to an equivalent norm) for $\F\in \{H_p,B_{p,q}\}$, i.e. there exists $C_{1}>0$ depending only on $s$, $p$, $q$, and $n$, such that
$$
\|uv\|_{\F^s(\M)}\leq C_{1}\|u\|_{\F^s(\M)}\|v\|_{\F^s(\M)}, \quad \text{for all} \quad u,v\in \F^s(\M).
$$
In addition, $\F^s(\M)$ is closed under holomorphic functional calculus, that is if $v\in \F^s(\M)$ and $f$ is an analytic function in some neighborhood of $\mathrm{Ran}(v)=\{v(\p)\in \mathbb{C}\, |\, \p\in \M\}$, then $f(v)\in \F^s(\M)$. Furthermore, if $U$ is a bounded set in $\F^s(\M)$ consisting of functions $u$ satisfying $\mathrm{Re}(u)\geq c$, for certain $c>0$ depending on $U$, then the set $\{u^{-1}\, :\, u\in U\}$ is also bounded in $\F^s(\M)$.
\end{prop}
\begin{proof}
(i) The embeddings~\eqref{Sobolev embedding 0} are shown in \cite[Theorem~7.4.2 (2) and (5)]{TriebelbookII} for scalar functions and the proof for tensor-valued case is similar.
The embeddings~\eqref{Sobolev embedding} follow from \cite[Theorem~14.2]{Ama13} and \eqref{Sobolev embedding 0}. 

(ii) \eqref{pointwise mul 0} follows from \cite[Theorem 9.2]{Ama13} by choosing the weight function $\rho\equiv 1$. 
Then \eqref{pointwise mul 1} is a direct consequence of \eqref{Sobolev embedding} and \eqref{pointwise mul 0}.

(iii) The fact that $\F^s(\M)$ is a Banach algebra is a direct consequence of \cite[Theorem 9.3]{Ama13} by choosing the weight function equal to one. For the closedness under the holomorphic functional calculus, by following the proof of \cite[Lemma 6.2]{RoSch}, let $ B_{j,R}=\B_\M(\p_j,R)$, $j\in\{1,\dots,N\}$ with $N\in\mathbb{N}$, be an open cover of $\M$ consisting of geodesic balls of radius $R>0$, centred at $\p_{j}\in\M$.
Moreover, let $\phi_{j}$, $j\in\{1,\dots,N\}$, be a subordinated partition of unity and assume that the closure of each $B_{j,3R/2}$, $j\in\{1,\dots,N\}$, is contained in a single coordinate chart. Take $u\in \F^s(\M)$ and assume that $u$ is pointwise invertible. Let $\omega:\mathbb{R}:\rightarrow [0,1]$ be a smooth nonincreasing function that equals $1$ on $[0, 1/2]$ and $0$ on $[3/4, +\infty)$, and define
$$
u_{j}(\p)=\omega\Big(\frac{d(\p,\p_j)}{2R}\Big)u(\p) + \Big(1-\omega\Big(\frac{d(\p,\p_j)}{2R}\Big)\Big)u(\p_j), \quad \p\in\M, \, j=1,\dots ,N.
$$
Choose $R$ sufficiently small such that 
$$
\|u(\cdot)-u(\p_{j})\|_{\infty}\leq \frac{1}{2}|u(\p_{j})|, \quad \text{for each} \quad j\in\{1,\dots,N\}.
$$
For each $j\in\{1,\dots,N\}$ the push-forward of $\eta_{j}=u_{j}-u(\p_{j})$ belongs to $\F^s(\mathbb{R}^{n})$. By \cite[Theorems 6 and 10]{BuSi}, so does the push-forward of $\eta_{j}(u(\p_{j})+\eta_{j})^{-1}$. Therefore, $\eta_{j}(u(\p_{j})+\eta_{j})^{-1}$ belongs to $\F^s(\M)$. We have 
$$
u_{j}^{-1}=(u(\p_{j})+\eta_{j})^{-1}=\frac{1}{u(\p_{j})}(1-\eta_{j}(u(\p_{j})+\eta_{j})^{-1}), \quad j\in\{1,\dots,N\},
$$
so that $u_{j}^{-1}\in \F^s(\M)$ for each $j$. Then, the identity 
$$
1=\sum_{j=1}^{N}\phi_{j}=\Big(\sum_{j=1}^{N}\phi_{j}u_{j}^{-1}\Big)u
$$
shows that $u^{-1}\in \F^s(\M)$ as well; recall here that $\phi_{j}u=\phi_{j}u_{j}$ for each $j$. The closedness under holomorphic functional calculus follows immediately by the expression 
$$
f(u)=\frac{1}{2\pi i}\int_{\Gamma}f(-\lambda)(u+\lambda)^{-1}d\lambda,
$$
where $\Gamma$ is a finite simple path around $\mathrm{Ran}(-u)$, in the area of holomorphicity of $f$. The boundedness of the set $\{u^{-1}\, |\, u\in U\}$ follows by the above construction.
\end{proof}

\subsection{Functional analytic tools}\label{Section 2.2}

We introduce some tools from function analysis. 
The reader may refer to the treatises \cite{Ama95}, \cite{DenHiePru03} and \cite{PruSim16} for more details of these concepts.

Let $X_{1}\overset{d}{\hookrightarrow}X_{0}$ be a continuously and densely injected complex Banach couple.
\begin{definition}[Dissipativity] 
A linear operators $\cA$ in $X_0$ with $\dom(\cA)=X_1$ is called {\em dissipative} if for all $\lambda>0$ and $x\in X_1$
$$
\|(\lambda -\cA)x\|\geq \lambda \|x\|.
$$
\end{definition}

\begin{definition}[Sectoriality]\label{secdef}
Let $\mathcal{P}(K,\theta)$, $K\geq1$, $\theta\in[0,\pi)$, be the class of all closed densely defined linear operators $\cA$ in $X_{0}$ such that 
$$
\Sigma_{\theta}:=\{\lambda\in\mathbb{C}\, : \, |\arg(\lambda)|\leq\theta\}\cup\{0\}\subset\rho{(-\cA)} 
$$
and
$$
(1+|\lambda|)\|(\cA+\lambda)^{-1}\|_{\mathcal{L}(X_{0})}\leq K, \quad \lambda\in \Sigma_{\theta}.
$$
The elements in $\mathcal{P}(\theta)=\bigcup\limits_{K\geq1}\mathcal{P}(K,\theta)$ are called {\em invertible sectorial operators of angle $\theta$} and for each $\cA\in\mathcal{P}(\theta)$ the constant $\inf\{K\, : \, \cA\in \mathcal{P}(K,\theta)\}$ is called {\em the sectorial bound of $\cA$}. 

Furthermore, denote by $\mathcal{S}(K,\theta)$ the supclass of $\mathcal{P}(K,\theta)$ such that if $\cA\in\mathcal{S}(K,\theta)$ then
$$
\Sigma_{\theta}\backslash\{0\}\subset\rho{(-\cA)} \quad \mbox{and} \quad |\lambda|\|(\cA+\lambda)^{-1}\|_{\mathcal{L}(X_0)}\leq K, \quad \lambda\in \Sigma_{\theta}\backslash\{0\}.
$$
The elements in $\mathcal{S}(\theta)=\bigcup\limits_{K\geq1}\mathcal{S}(K,\theta)$ are called {\em sectorial operators of angle $\theta$} and for each $\cA\in\mathcal{S}(\theta)\backslash\mathcal{P}(\theta)$ the constant $\inf\{K\, :\, \cA\in \mathcal{S}(K,\theta)\}$ is called {\em the sectorial bound of $\cA$}. 
\end{definition}

Recall that $\mathcal{P}(K,\theta)\subset \mathcal{P}(2K+1,\phi)$ for some $\phi\in(\theta,\pi)$, see, e.g. \cite[(III.4.6.4)-(III.4.6.5)]{Ama95}, and similarly for the class $\mathcal{S}(\theta)$. Hence, whenever $\cA\in \mathcal{P}(\theta)$ or $\cA\in \mathcal{S}(\theta)$ we can always assume that $\theta>0$. Moreover, for any $\rho\geq0$ and $\theta\in(0,\pi)$, let the counterclockwise oriented path 
$$
\Gamma_{\rho,\theta}=\{re^{\pm i\theta}\in\mathbb{C}\,:\,r\geq\rho\}\cup\{\rho e^{i\phi}\in\mathbb{C}\,:\,\theta\leq\phi\leq2\pi-\theta\}.
$$
The holomorphic functional calculus for sectorial operators in the class $\mathcal{P}(\theta)$ is defined by the Dunford integral formula, see, e.g. \cite[Theorem 1.7]{DenHiePru03}. A typical example are the complex powers; for $\mathrm{Re}(z)<0$ they are defined by
\begin{equation}\label{cp}
\cA^{z}=\frac{1}{2\pi i}\int_{\Gamma_{\rho,\theta}}(-\lambda)^{z}(\cA+\lambda)^{-1}d\lambda,
\end{equation}
where $\rho>0$ is sufficiently small. The family $\{\cA^{z}\}_{\mathrm{Re}(z)<0}$ together with $\cA^{0}=I$ is a strongly continuous analytic semigroup on $X_{0}$, see, e.g. \cite[Theorem III.4.6.2 and Theorem III.4.6.5]{Ama95}. Moreover, each $\cA^{z}$, $\mathrm{Re}(z)<0$, is an injection and the complex powers for positive real part $\cA^{-z}$ are defined by $\cA^{-z}=(\cA^{z})^{-1}$, see, e.g. \cite[(III.4.6.12)]{Ama95}. By Cauchy's theorem we can deform the path in \eqref{cp} and define the imaginary powers $\cA^{it}$, $t\in\mathbb{R}\backslash\{0\}$, as the closure of the operator
$$
\cA^{it}=\frac{\sin(i\pi t)}{i\pi t}\int_{0}^{+\infty}s^{it}(\cA+s)^{-2}\cA \, ds \quad \text{in}\quad \dom(\cA),
$$
see, e.g. \cite[(III.4.6.21)]{Ama95}. For the properties of the complex powers of sectorial operators, we refer to \cite[Theorem III.4.6.5]{Ama95}. Concerning the imaginary powers, the following property can be satisfied.

\begin{definition}[Bounded imaginary powers] Let $\cA\in\mathcal{P}(0)$ in $X_{0}$ and assume that there exist some $\delta,M>0$ such that $\cA^{it}\in \mathcal{L}(X_{0})$ and $\|\cA^{it}\|_{\mathcal{L}(X_{0})}\leq M$ when $t\in(-\delta,\delta)$. Then, $\cA^{it}\in \mathcal{L}(X_{0})$ for each $t\in\mathbb{R}$ and there exist some $\phi,\widetilde{M}>0$ such that $\|\cA^{it}\|_{\mathcal{L}(X_{0})}\leq \widetilde{M}e^{\phi|t|}$, $t\in\mathbb{R}${\em ;} in this case we say that {\em $\cA$ has bounded imaginary powers} and denote it by $\cA\in\mathcal{BIP}(\phi)$.
\end{definition}

The following property, stronger than the boundedness of the imaginary powers, can also be satisfied by operators in the class $\mathcal{P}(\theta)$.

\begin{definition}[Bounded $H^{\infty}$-calculus]
Let $\theta\in(0,\pi)$, $\phi\in[0,\theta)$, $\cA\in\mathcal{P}(\theta)$ and let $H_{0}^{\infty}(\phi)$ be the space of all bounded holomorphic functions $f:\mathbb{C}\backslash \Sigma_{\phi}\rightarrow \mathbb{C}$ satisfying 
$$
|f(\lambda)|\leq c \Big(\frac{|\lambda|}{1+|\lambda|^{2}}\Big)^{\eta}\quad \text{for any} \quad \lambda\in \mathbb{C}\backslash \Sigma_{\phi} 
$$
and some $c,\eta>0$ depending on $f$.
Any $f\in H_{0}^{\infty}(\phi)$ defines an element $f(-\cA)\in \mathcal{L}(X_{0})$ by 
\begin{align*}
f(-\cA)=\frac{1}{2\pi i}\int_{\Gamma_{\theta}}f(\lambda)(\cA+\lambda)^{-1} d\lambda.
\end{align*}
We say that the operator $\cA$ {\em has bounded $H^{\infty}$-calculus of angle $\phi$}, and we denote by $\cA\in \mathcal{H}^{\infty}(\phi)$, if there exists some $C>0$ such that
\begin{equation*}
\|f(-\cA)\|_{\mathcal{L}(X_{0})}\leq C\sup_{\lambda\in\mathbb{C}\backslash \Sigma_{\phi}}|f(\lambda)| \quad \mbox{for any} \quad f\in H_{0}^{\infty}(\phi).
\end{equation*}
\end{definition}

\begin{definition}[$R$-boundedness]\label{rsec}
A set $E\subset \mathcal{L}(X_{0})$ is called {\em $R$-bounded} if for every $T_{1},\dots,T_{N}\in E$ and $x_{1},\dots,x_{N}\in X_0$, $N\in \N$, we have
\begin{eqnarray*}
\|\sum_{k=1}^{N}\epsilon_{k}T_{k}x_{k}\|_{L_{2}((0,1),X_0)} \leq C \|\sum_{k=1}^{N}\epsilon_{k}x_{k}\|_{L_{2}((0,1),X_0)} 
\end{eqnarray*}
for certain $C>0$, where $\{\epsilon_{k}\}_{k=1}^{\infty}$ is the sequence of Rademacher functions. The infimum of all such constants $C>0$ is called {\em the $R$-bound of $E$}.

Denote by $\mathcal{R}(\theta)$, $\theta\in[0,\pi)$, the class of all operators $\cA\in \mathcal{S}(\theta)$ in $X_{0}$ such that the set $E=\{\lambda(\cA+\lambda)^{-1}\, :\, \lambda\in \Sigma_{\theta}\backslash\{0\}\}$ is $R$-bounded. If $\cA\in \mathcal{R}(\theta)$ then $\cA$ is called {\em $R$-sectorial of angle $\theta$} and the $R$-bound of $E$ is called {\em the $R$-sectorial bound of $\cA$}. 
\end{definition} 

Given any $T>0$, recall the embedding
\begin{equation}\label{inembmaxreg}
 L_q((0,T),X_1)\cap H^1_q((0,T),X_0) \hookrightarrow C([0,T],X_{1/q,q}),
\end{equation} 
cf. \cite[Theorem III.4.10.2]{Ama95}, and $X_{1/q,q}:= (X_0,X_1)_{1-1/q,q}$. 
Let $ J=(0,T) $ and
\begin{align*}
{\bE}_0(J) := L_q(J,X_0), \quad 
{\bE}_1(J) := L_q(J,X_1)\cap H^1_q(J,X_0) .
\end{align*}
For any $\cA\in \S(\theta)$, $\theta\in (\pi/2,\pi)$, with $\dom(\cA)=X_1$, 
$$
\cA\in \mathcal{MR}_q(X_1, X_0) 
$$
holds iff
$$
\displaystyle ( \partial_{t} + \cA, \gamma_0)\in \Lis ({\bE}_1(J),{\bE}_0(J)\times X_{1/q,q} ) ,
$$
where $\gamma_0$ is the trace map at $0$, i.e. $\gamma_{0}(u)=u(0)$.
 
If we restrict to the class of UMD (unconditionality of martingale differences property, see, e.g. \cite[Section III.4.4]{Ama95}) Banach spaces, then we have the following.

\begin{theorem}[{\rm Kalton and Weis, \cite[Theorem 6.5]{KaW} or \cite[Theorem 4.2]{Weis2}}]\label{KaWeTh}
If $X_{0}$ is UMD and $\cA\in\mathcal{R}(\theta)$ in $X_{0}$ with $\dom(\cA)=X_1$ and $\displaystyle \theta\in (\pi/2,\pi)$, then $\cA \in \mathcal{MR}_q(X_1,X_0)$ for all $q\in (1,\infty)$. 
\end{theorem}

\section{Imaginary powers of elliptic operators}\label{Section 3}

Let $\sL_{s,\tau}^\eta$ denote the $H^s_p(\M,V)$-realization of $\sL$, where $\sL$ is defined in \eqref{def of sL}. Note that when $\eta+\tau>0$ and $\a\equiv 1$, $ \sL_{s,\tau}^\eta$ is the Bochner Laplacian, cf. \cite[Example 1.6]{Ama17}. 
In the sequel, we will omit the indices $ \eta,\tau $ and write $\sL_s $ whenever the choice of $V$ is clear from the context.
If, further, the choice of $s$ is immaterial, we will simply use $\sL$. 

The sectoriality of $\sL $ can be shown by modifying an argument in E. Davies \cite{Dav89}. This idea has been adopted in \cite{Shao16} to show the sectoriality of a class of singular operators acting on scalar function. 

\begin{theorem}\label{Thm: sectoriality Lap}
$\sL_{s,\tau}^\eta \in \cS(\theta)$ for some $\theta>\pi/2$.
\end{theorem}
\begin{proof}
Choosing a local orthonormal frame of vector fields $\{e_j\}_{j=1}^n$ satisfying $\nabla_{e_k}e_j=0$ and with dual covector fields $\{e^j\}_{j=1}^n$. We set
$$
e_{(i)}=e_{i_1}\otimes \cdots \otimes e_{i_r}, \quad e^{(i)}=e^{i_1}\otimes \cdots \otimes e^{i_r},
$$ 
where $(i)=(i_1,\dots,i_r)\in \mathbb{J}^r:=\{1,\dots,n\}^r$ for any $r\in \Nz$. We express any $a\in V^{\eta+\tau}_{\tau+\eta} $ by 
$$
a=a^{(i)(s)}_{(j)(r)} e_{(i)(s)}\otimes e^{(j)(r)}
$$
where $(i),(r)\in \mathbb{J}^\eta$, $(j),(s)\in \mathbb{J}^\tau$ and $(i)(s)=(i_1,\dots,i_{\eta} , s_1,\dots,s_{\tau})$ and $(j)(r)=(j_1,\dots,j_{\tau} , r_1,\dots,r_{\eta})$. 

It follows from \cite[Appendix C, Equation~(1.35)]{Taylor2} that
\begin{align}
\notag \sL u= -\sum_{j=1}^n \left[ \a \nabla_{e_j} \nabla_{e_j} u + {\rm div}(\a e_j)\nabla_{e_j} u \right] 
=& - \a g^*\otimes \mathfrak{I} \cdot\nabla^2 u - \sum_{j=1}^n \nabla_{e_j} \a \nabla_{e_j} u \\
\label{exp of L}
=& - \a g^*\otimes \mathfrak{I} \cdot\nabla^2 u - \nabla \a \otimes \mathfrak{I} \cdot \nabla u 
\end{align} 
where ${\rm div}$ is the divergence operator defined in \cite[Chapter 10, Equation~(1.39)]{Taylor2} and $\mathfrak{I}\in \mathcal{T}^{\eta+\tau}_{\tau+\eta}\M$ is defined by $\mathfrak{I}^{(i)(s)}_{(j)(r)}=\delta^{(i)(s)}_{(j)(r)}$ with
$$
\delta^{(i)(s)}_{(j)(r)}=
\begin{cases}
1 \quad& \text{if } (i)=(r), (j)=(s) \\
0 & \text{otherwise}.
\end{cases}
$$
Further, the notation $\cdot$ denotes the complete contraction, cf. \cite[p. 4]{Ama17} and $g^*$ is the covariant metric induced by $g$ on $T^*\M$.
Then, it follows from \cite[Theorem~1.30 (i)]{Ama17} that for $k\in \N$, $\sL_{2k,\tau}^\eta$ is a closed operator with $\dom(\sL_{2k,\tau}^\eta)=H^{2k+2}_p(\M,V)$.
By the interpolation theory and \eqref{defofHspace}, we conclude that $\sL_{s,\tau}^\eta$ is a closed operator with $\dom(\sL_{s,\tau}^\eta)=H^{s+2}_p(\M,V)$ for all $s\geq 0$.


Step 1: $s=0$ and $p=2$

Since the $L_2(\M,V)$-realization of $ \sL $ is self-adjoint and positive semi-definite, we infer that $\sL_0$ is dissipative.
In particular, $-\sL_0 |_{H^2_2(M,V)} \in \cS(\theta)$ for any $\theta\in (0,\pi)$, see, e.g. \cite[Theorem III.4.6.7]{Ama95}, which implies that it generates an analytic semigroup on $L_2(\M,V)$.
Moreover, the Lumer-Phillips' theorem implies that this semigroup is an $L_2$-contraction.

$\sL_0 $ is associated with a quadratic form $\fb: H^1_2(\M,V)\times H^1_2(\M,V) \to \C$ defined by
$$
\fb(u,v)= (\a\nabla u | \nabla \overline{v} ) _g ,\quad u,v \in H^1_2(\M,V).
$$
In the rest of Step 1, we will follow the idea in the proof of \cite[Theorem~2.5]{Ouh92}.
First, direct computations show that $(|u|_g-1)^+\sg u \in H^1_2(\M,V) $ whenever $u\in H^1_2(\M,V)$,
where
\begin{align*}
\sg u:=
\begin{cases}
u/|u|_g, \quad & u\neq 0;\\
0, &u=0.
\end{cases}
\end{align*}
Choosing a local orthonormal frame of vector fields $\{e_j\}_{j=1}^n$ with $\nabla_{e_k}e_j=0$, then locally it holds that
\begin{align}\label{der of norm}
\nabla_{e_l} |u|_g= \frac{\nabla_{e_l} ( u | \overline{ u})_g}{2|u|_g} = \frac{( \nabla_{e_l} u | \overline{ u})_g}{2|u|_g} +\frac{ ( u | \overline{ \nabla_{e_l} u})_g}{2|u|_g} = \re \frac{( \nabla_{e_l} u | \overline{ u})_g }{|u|_g}. 
\end{align}
Therefore, 
$
\displaystyle \nabla |u|_g =\re \frac{\nabla u \cdot \overline{u} }{|u|_g}.
$
When $|u|_g \geq 1$, it follows from \eqref{der of norm} that
\begin{align*}
\nabla_{e_l} \left[(|u|_g-1)\sg u \right]=& \nabla_{e_l} u - \nabla_{e_l} \left( \frac{u}{|u|_g} \right) 
= \frac{ |u|_g-1}{|u|_g} \nabla_{e_l} u + \frac{1}{|u|_g^2}u \nabla_{e_l} |u|_g \\
=& \frac{ |u|_g-1}{|u|_g} \nabla_{e_l} u + u \re \frac{( \nabla_{e_l} u | \overline{ u})_g }{|u|_g^3}.
\end{align*}
We thus infer that
\begin{align*}
\nabla \left[(|u|_g-1)\sg u \right]
= \frac{ |u|_g-1}{|u|_g} \nabla u + \frac{ \re (\nabla u \cdot \overline{u}) \otimes u }{|u|_g^3} .
\end{align*}
This implies that when $|u|_g \geq 1$
\begin{align*}
 \re \Big( \a \nabla u | \nabla [(|u|_{g}-1)\overline{\sg u} ] \Big)_g 
= \a |\nabla u |_g^2 |\left(\frac{ |u|_g-1}{|u|_g} \right) + \a \frac{| \re (\nabla u \cdot \overline{u})|_g^2}{|u|_g^3} \geq 0.
\end{align*}
Therefore
$$
\re \fb(u , (|u|_g-1)^+\sg u ) \geq 0.
$$
Denote by $\langle \cdot , \cdot \rangle $ the inner product of $L_2(\M,V)$.
Given any $\lambda>0$, let $u=\lambda(\lambda+\sL )^{-1} f$ in the above inequality for some $f\in L_2(\M,V)\cap L_\infty(\M,V)$ with $\|f\|_\infty\leq 1$. Then
\begin{align*}
0&\leq \, \re \langle \sL [\lambda(\lambda+\sL )^{-1} f] , (|\lambda(\lambda+\sL)^{-1} f|_g-1 )^+\sg \overline{\lambda(\lambda+\sL )^{-1} f} \rangle \\
=& \lambda \re \!\! \int_\M \! [ (f | \sg \overline{\lambda(\lambda+\sL )^{-1} f} )_g \! - \! |\lambda(\lambda+\sL )^{-1} f|_g ] 
(|\lambda(\lambda+\sL )^{-1} \! f|_g \!- \! 1 )^+ d\mu_g.
\end{align*}
But $\|f\|_\infty\leq 1$ implies that $|(f | \sg \overline{\lambda(\lambda+\sL )^{-1} f} )_g|\leq 1$ a.e. and thus
\begin{equation}
\label{contraction eq}
\re \Big\{ \left[(f | \sg \overline{\lambda(\lambda+\sL )^{-1} f} )_g - |\lambda(\lambda+\sL)^{-1} f|_g \right] (|\lambda(\lambda+\sL )^{-1} f|_g-1 )^+ \Big\} \leq 0 
\end{equation}
holds a.e. when $|\lambda(\lambda+\sL)^{-1} f|_g>1$. When $|\lambda(\lambda+\sL)^{-1} f|_g\leq 1$, \eqref{contraction eq} clearly holds true.
We thus conclude that
$$
|\lambda(\lambda+\sL)^{-1} f|_g \leq 1 \quad \text{a.e.}.
$$
From the standard semigroup theory, it follows that
$$
e^{-t\sL } u= \lim\limits_{n\to \infty} \left[\frac{n}{t}\left(1+\frac{n}{t} \sL \right)\right]^{-n}u ,\quad u \in L_\infty(\M,V)\cap L_2(\M,V).
$$
We thus infer the $L_\infty$-contraction of the semigroup $\{e^{-t\sL }\}_{t\geq 0}$, i.e.
$$
\|e^{-t\sL } u \|_\infty \leq \|u\|_\infty ,\quad u \in L_\infty(\M,V)\cap L_2(\M,V).
$$

Step 2: $s=0$ and $p\in (1,\infty)$

The proof follows a classic idea in \cite[Chapter~1.4]{Dav89}, which was originally presented for scalar functions. By a duality argument, we can prove 
$$
\|e^{-t\sL } u \|_1 \leq \|u\|_1 ,\quad u \in L_1 (\M,V).
$$
Then, the Riesz-Thorin interpolation theorem implies that
$$
\|e^{-t\sL } u \|_p \leq \|u\|_p ,\quad u \in L_p(\M,V).
$$
When $(\M,g)$ is a closed manifold, the H\"older's inequality and the strong continuity of $\{e^{-t\sL}\}_{t\geq 0}$ in $L_2(\M,V)$ show that for all $ u \in L_2(\M,V) $
\begin{align*}
\lim\limits_{t\to 0^+} \| e^{-t\sL } u - u \|_1 \leq \lim\limits_{t\to 0^+} \| e^{-t\sL } u - u \|_2 ({\rm vol}(\M))^{1/2} =0.
\end{align*}
Since $L_2(\M,V) $ is dense in $L_1(\M,V) $, we thus obtain the strong continuity of 
$\{e^{-t\sL}\}_{t\geq 0}$ in $L_1(\M,V)$.

When $(\M,g)=(\R^n,g_n)$, following the proof of \cite[Theorem~1.4.1]{Dav89}, we can show that $\{e^{-t\sL}\}_{t\geq 0}$ is strongly continuous in $L_1(\M)$.
Note that for every $u=(u_1,\dots, u_n)\in L_1(\M,V)$ with $u_j\in L_1(\M)$
$$
e^{-t\sL}u=(e^{-t\sL} u_1,\dots, e^{-t\sL}u_n).
$$
This implies the strong continuity of 
$\{e^{-t\sL}\}_{t\geq 0}$ in $L_1(\M,V)$.

By the interpolation theory, we get the strong continuity in $L_p(\M,V)$ for $1<p<2$ and a standard duality argument yields the same for $2<p<\infty$. 
Then we can follow the Stein interpolation argument in \cite[Theorem~1.4.2]{Dav89} and prove that $\{e^{-t\sL}\}_{t\geq 0}$ can be extended to an analytic semigroup on $L_p(\M,V)$ in a sector $\Sigma_\phi$ with
$$
\phi \geq \frac{\pi}{2}\left(1- \left |\frac{2}{p}-1 \right| \right),\quad p\in (1,\infty).
$$
By the standard semigroup theory, this implies $\sL_0 \in \cS(\theta)$ with $\theta>\frac{\pi}{2}$.
 
Step 3: $s>0$

The proof for this case follows by an analogous argument as for the case $s>0$ and $1<p<\infty$ in the proof of \cite[Theorem~5.1]{RoidosShao}. 

First we will show that
\begin{equation}\label{resolvreestr}
 \Sigma_{\theta} \setminus \{0\} \subset \rho(-\sL_s) \quad \text{and} \quad (\lambda+\sL_0)^{-1}|_{H^{s}_p(\M,V)}=(\lambda+\sL_s)^{-1}, \quad \lambda\in \Sigma_{\theta}\setminus \{0\} ,
\end{equation}
where $\theta$ is the sectorial angle of $ \sL_0 $ asserted in Step 2. It is sufficient to verify the identities
$$
(\lambda+\sL_0)^{-1}(\lambda+\sL_s)=I \quad \text{and} \quad (\lambda+\sL_s)(\lambda+\sL_0)^{-1}=I, \quad \lambda\in \Sigma_{\theta} \setminus \{0\} ,
$$
on $H^{s+2}_p(\M,V)$ and $H^{s}_p(\M,V)$, respectively. The first one is trivial. For the second one, let $u\in H^{2}_p(\M,V)$ such that $(\lambda+\sL)u\in H^{s}_p(\M,V)$. If $s\in (0,2]$, then we have that $u,\sL u\in H^{s}_p(\M,V)$, i.e. $u$ belongs to the domain of $\sL$ in $H^{s}_p(\M,V)$, which implies that $u\in H^{s+2}_p(\M,V)$. The higher values of $s$ can be treated by iteration. 

Step 3a: $s\in 2\mathbb{N}$

We proceed by induction. Assume that the result holds for some $s\in \mathbb{N}$. For each $v\in H^{s+2}_p(\M,V)$, we have 
\begin{eqnarray*}
\lefteqn{\|\lambda(\lambda+\sL_{s+2})^{-1}v\|_{H^{s+2}_p(\M,V)}\,\,\,=\,\,\,\|\lambda(\lambda+\sL_{s})^{-1}v\|_{H^{s+2}_p(\M,V)}}\\
&\leq&C_{1}\Big(\|\lambda(\lambda+\sL_{s})^{-1}v\|_{H^{s}_p(\M,V)}+\|\sL_{s}(\lambda(\lambda+\sL_{s})^{-1}v)\|_{H^{s}_p(\M,V)}\Big)\\
&=&C_{1}\Big(\|\lambda(\lambda+\sL_{s})^{-1}v\|_{H^{s}_p(\M,V)}+\|\lambda(\lambda+\sL_{s})^{-1}\sL_{s}v\|_{H^{s}_p(\M,V)}\Big)\\
&\leq&C_{2}\Big(\|v\|_{H^{s}_p(\M,V)}+\|\sL_{s}v\|_{H^{s}_p(\M,V)}\Big) \leq C_{3}\|v\|_{H^{s+2}_p(\M,V)},
\end{eqnarray*}
for certain $C_{1},C_{2},C_{3}>0$ independent of $\lambda \in \Sigma_{\theta} \setminus \{0\} $. 

Step 3b: $s\in\mathbb{R}$

The results follows by Proposition~\ref{Prop: interpolation} and the interpolation theory. More precisely, for each $s\in (k,k+2)$, $k\in\Nz$, and each $\lambda\in \Sigma_{\theta}$, we have
$$
\lambda(\lambda+\sL )^{-1}\in \mathcal{L}(H^{k}_p(\M,V)) \quad \text{and} \quad \lambda(\lambda+\sL )^{-1}\in \mathcal{L}(H^{k+2}_p(\M,V)),
$$
with norm independent of $\lambda\in \Sigma_{\theta} \setminus \{0\} $.
So the required estimate is obtained by \cite[Theorem 2.6]{Lunar18} and Proposition~\ref{Prop: interpolation}. 
\end{proof}

\begin{remark}\label{bipandrsec}
For any $p\in (1,\infty)$, $\theta\in [0,\pi)$ and $\phi>0$ there exists a $c>0$ such that $c+\sL_0 \in \mathcal{R}(\theta)\cap\mathcal{BIP}(\phi)$. This follows by \cite[Corollary 10.4]{AmaHieSim94} in combination with \cite[Theorem 4]{ClPr}. 
Note that, by \eqref{exp of L}, the symbol of $\sL_0$ is defined by 
$$
[b\mapsto \a g^*\otimes \mathfrak{I}\cdot (\xi^{\otimes 2} \otimes b)]=[b\mapsto \a |\xi|_g^2 b ], \quad b\in V, \, \xi\in T^*\M.
$$ 
Therefore, $ \sL_0$ is $\phi$-elliptic in the sense of \cite[Theorem 10.3]{AmaHieSim94}, for arbitrary small $\phi>0$.
\end{remark}

\begin{prop}[$\mathcal{BIP}$ for higher $s$]\label{Prop: BIP higher order}
There exists a $c>0$ with the following property: for any $s\geq0$ and any $\phi>0$, we have $c+\sL_{s}\in \mathcal{BIP}(\phi)$.
\end{prop}
\begin{proof}
We proceed by induction and interpolation. For $s=0$, the result holds true due to Remark \ref{bipandrsec}. Assume that the statement holds for certain $s\in \Nz$. Let $c>0$ be as in Remark \ref{bipandrsec}. 
By Remark \ref{bipandrsec} and the argument leading to \eqref{resolvreestr}, for any $\varepsilon>0$ and any $t\in\mathbb{R}$, we have that
\begin{equation}\label{frecpowrestr}
(c+\sL_{\nu})^{-\varepsilon+it}=(c+\sL_{\widetilde{\nu}})^{-\varepsilon+it}|_{H^{\nu}_{p}(\M,V)},
\end{equation}
where $\nu\geq \widetilde{\nu}$. Let $u\in H^{s+2}_{p}(\M,V)$. By the boundedness of the imaginary powers of $c+\sL_{s}$, in particular by \cite[Lemma III.4.7.4 (ii)]{Ama95}, we have that
\begin{eqnarray*}
\lefteqn{\|(c+\sL_{s+2})^{-\varepsilon+it}u\|_{H^{s+2}_{p}(\M,V)}\,\,\,=\,\,\,\|(c+\sL_{s})^{-\varepsilon+it}u\|_{H^{s+2}_{p}(\M,V)}}\\
&\leq&C_{1}\Big(\|(c+\sL_{s})^{-\varepsilon+it}u\|_{H^{s}_{p}(\M,V)}+\|(c+\sL_{s})^{-\varepsilon+it} \sL_{s}u\|_{H^{s}_{p}(\M,V)}\Big)\\
&\leq&C_{1}\|(c+\sL_{s})^{-\varepsilon+it}\|_{\mathcal{L}(H^{s}_{p}(\M,V))}\Big(\|u\|_{H^{s}_{p}(\M,V)}+\|\sL_{s}u\|_{H^{s}_{p}(\M,V)}\Big)\\
&\leq&C_{2}e^{\phi|t|}\|u\|_{H^{s+2}_{p}(\M,V)},
\end{eqnarray*}
for certain $C_{1},C_{2}>0$ independent of $\varepsilon$ and $t$. Hence, from \cite[Lemma III.4.7.4~(i)]{Ama95}, we deduce that $c+\sL_{s+2}\in \mathcal{BIP}(\phi)$. 

By \cite[Theorem 2.6]{Lunar18}, Proposition~\ref{Prop: interpolation} and \eqref{frecpowrestr}, for each $\rho\in (0,1)$ we have
\begin{eqnarray*}
\lefteqn{\|(c+\sL_{s+\rho})^{-\varepsilon+it}\|_{\mathcal{L}(H^{s+\rho}_{p}(\M,V))}}\\
& \leq & C_3 \Big(\|(c+\sL_{s})^{-\varepsilon+it}\|_{\mathcal{L}(H^{s}_{p}(\M,V))}\Big)^{1-\rho}\Big(\|(c+\sL_{s+2})^{-\varepsilon+it}\|_{\mathcal{L}(H^{s+2}_{p}(\M,V))}\Big)^{\rho},
\end{eqnarray*} 
for certain $C_3>0$ independent of $\varepsilon$ and $t$. Hence, again by \cite[Lemma III.4.7.4 (ii)]{Ama95}, we obtain that
$$
\|(c+\sL_{s+\rho})^{-\varepsilon+it}\|_{\mathcal{L}(H^{s+\rho}_{p}(\M,V))}\leq C_4 (e^{\phi|t|})^{1-\rho}(e^{\phi|t|})^{\rho},
$$
for some $C_4>0$ independent of $\varepsilon$ and $t$. The result then follows by \cite[Lemma III.4.7.4 (i)]{Ama95}.
\end{proof}

\section{The Fractional Powers of $\sL$}\label{Section 4}

Following the discussion in \cite[Section~4]{RoidosShao2}, we can show for any $c\geq 0$ 
\begin{align}
\label{S4: fractional Delta def}
J^\sigma_c u:&=\frac{\sin(\pi\sigma)}{\pi} \int_0^\infty x^{\sigma-1} ( c+\sL^\eta_{s,\tau})(x+c+\sL^\eta_{s,\tau})^{-1} u\, dx
\end{align}
is well-defined for all $u\in H^{s+2}_p(\M,V)=\dom(\sL^\eta_{s,\tau})$. 
Indeed, we have proved that a formula similar to \eqref{S4: fractional Delta def}, \cite[(4.1)]{RoidosShao2}, holds true for $c=0$ and an operator $-\underline{\Delta}_{F,p}$. See \cite[pp. 15-17]{RoidosShao2}. The proof only relies on the fact that $-\underline{\Delta}_{F,p} \in \cS(\theta)$ for some $\theta>0$.

Note that \eqref{S4: fractional Delta def} is exactly Balakrishnan's formula for fractional powers of dissipative operators.
By \cite[(2.7)]{Bal60},
\begin{center}
$( c+\sL^\eta_{s,\tau})^\sigma$ is the smallest closed extension of $J^\sigma_c $.
\end{center} 
Therefore, \eqref{S4: fractional Delta def} converges for all $u\in \dom((c +\sL^\eta_{s,\tau})^\sigma)$ in $H^s_p(\M,V)$. 
The domain $\dom((c+\sL^\eta_{s,\tau})^\sigma)$ is independent of $c\geq 0$, cf. \cite[Lemma~2.3.5]{Tan}.

Due to Proposition~\ref{Prop: BIP higher order}, for certain $c>0$, the operator $c+\sL^\eta_{s,\tau}$ has bounded imaginary powers.
By Proposition~\ref{Prop: interpolation}, \cite[(I.2.9.8)]{Ama95} and \cite[Lemma 2.3.5]{Tan}, we infer that
$$
\dom((\sL^\eta_{s,\tau})^\sigma) \doteq H^{s+2\sigma}_p(\M,V).
$$

\begin{prop} 
For any $s\geq0$ and any $\theta>0$, there exists a $c>0$ such that $ c + (\sL_s)^\sigma \in \mathcal{R}(\theta)$.
\end{prop}
\begin{proof}
We follow the ideas in Step 2 of the proof of \cite[Theorem 6.2]{RoidosShao}. Let $c_{1}>0$ be fixed and sufficiently large. By \cite[Theorem 1.1]{RoidosShao}, $(c+\sL_{s})^{\sigma}\in \mathcal{R}(\theta)$ for each $c\geq c_{1}$. Moreover, by \cite[Lemma 2.6]{RoSch} and the estimate in Part (i) in the proof of \cite[Theorem 1.1]{RoidosShao}, the $R$-sectorial bound of $( c+ \sL_{s})^{\sigma}$ is uniformly bounded in $c\geq c_{1}$. Let $\xi>1$ be fixed. Again by \cite[Lemma 2.6]{RoSch}, the operator $(c+\sL_{s})^{\sigma}+c^{\sigma+\xi}$ is $R$-sectorial and its $R$-sectorial bound can be chosen uniformly bounded in $c\geq c_{1}$. By \cite[(2.18)]{RoidosShao} we have
$$
\|((c+\sL_{s})^{\sigma}-(\sL_{s})^{\sigma})((c+\sL_{s})^{\sigma}+c^{\sigma+\xi})^{-1}\|_{\mathcal{L}(H^{s}_p(\M,V))}\leq C\frac{c^{\sigma}}{c^{\sigma+\xi}}
$$
for certain $C>0$ only depending on the sectorial bound of $(c+\sL_{s})^{\sigma}\in \mathcal{S}(0)$ and $\sigma$. By noting that
$$
(\sL_s)^{\sigma}+c^{\sigma+\xi}=(c+\sL_s)^{\sigma}+c^{\sigma+\xi}+(\sL_s)^{\sigma}-(c+\sL_{s})^{\sigma},
$$
after taking $c\geq c_{1}$ sufficiently large, we obtain the result by perturbation, see, e.g. \cite[Theorem 1]{KuWe}.
\end{proof}
 
\section{$L_q$-Maximal regularity}\label{Section 5}

Suppose that 
\begin{equation}\label{w-BC}
w\in BC^r(\M) \quad \text{for some}\quad r\in(0,\infty), 
\end{equation}
together with the following convention: if $\M=(\R^n, g_n)$, then we assume that there exists a constant $w_{\infty}>0$ such that 
\begin{equation}\label{winfinity}
\|w -w_{\infty}\|_{L_\infty(\R^n\setminus \B_\M(0,\widetilde{R}))}\rightarrow 0 \quad \text{as} \quad \widetilde{R}\rightarrow\infty.
\end{equation}
Assume, in addition, that there exists a constant $c_0>0$ such that
\begin{equation}\label{positive}
w>c_0.
\end{equation}

Let $s\in [0,r)$ and $1<p,q<\infty$. In this section, we will show that
$$
w \sL^\sigma\in \mathcal{MR}_q(H_p^{s+2\sigma}(\M,V), H^s_p(\M,V)).
$$
Let $f\in L_q(J, H^s_p(\M))$, where $J=(0,T)$ with $T>0$. 
Consider the Cauchy problem:
\begin{equation}
\label{MR-eq}
\left\{\begin{aligned}
\partial_t u +w \sL^\sigma u&=f ;\\
u(0)&=0 .
\end{aligned}\right.
\end{equation}
Our goal is to prove that \eqref{MR-eq} admits a unique solution
$$
u\in L_q(J, H_p^{s+2\sigma}(\M,V) ) \cap H^1_q(J, H^s_p(\M,V)).
$$
 
Let $R, \widetilde{R}>0$ and let $\widetilde{\omega} :\mathbb{R}\rightarrow [0,1]$ being a smooth non-increasing function that equals $1$ on $[0,1/2]$ and $0$ on $[3/4,\infty)$. Choose a finite open cover $U_{j}$ of $(\M,g)$, where $j\in\{1,\dots,N\}$ when $(\M,g)$ is a closed manifold and $j\in\{0,1,\dots,N\}$ when $(\M,g)=(\R^n,g_n)$, such that the following properties are fulfilled.

We let $U_{j}=B_{j,R}=\B_\M(\p_j,R)$ being geodesic balls with radius $R$ on $\M$ centered at $\p_j\in \M$, $j=1,\dots,N$. Moreover, we define 
\begin{equation}\label{localizerfunt}
w_{j,R}(\p)=\widetilde{\omega}\Big(\frac{d(\p,\p_j)}{2R}\Big)w(\p) + \Big(1-\widetilde{\omega}\Big(\frac{d(\p,\p_j)}{2R}\Big)\Big)w(\p_j), \,\, \p\in\M, \, j=1,\dots ,N.
\end{equation}
In the case of $\M=(\R^n, g_n)$, we assume that $U_{j}=B_{j,R}$, $j=1,\dots,N $, cover the closure of $\B_\M(0, \widetilde{R})$ and we further choose $U_{0}=\R^n\backslash \B_\M(0, \widetilde{R})$. In this case, we also define
\begin{equation}\label{localizerfuntforRm}
w_{0,\widetilde{R}}(\p)=\Big(1-\widetilde{\omega}\Big(\frac{d(\p,0)}{2\widetilde{R}}\Big)\Big)w(\p) + \widetilde{\omega}\Big(\frac{d(\p,0)}{2\widetilde{R}}\Big)w_{\infty}, \quad \p\in\M,
\end{equation}
where $w_{\infty}$ is defined in \eqref{winfinity}.

\begin{lem}\label{Lem: localization est}
Assume that $w$ satisfies \eqref{w-BC}-\eqref{winfinity} and $r\in (0,1]$. For any $\alpha\in [0,r)$ and $\varepsilon>0$, there exists an $R_{0}>0$ such that 
 $$
\|w_{j,R} -w(\p_j)\|_{\alpha,\infty}<\varepsilon \quad \text{for each} \quad j\in\{1,\dots,N\}, \quad \text{whenever}\quad R\in (0,R_0).
$$
In addition, if $\M=(\R^n, g_n)$, then there exists $\widehat{R}>0$ such that 
$$
\|w_{0,\widetilde{R}} -w_{\infty}\|_{\alpha,\infty}<\varepsilon, \quad \text{whenever}\quad \widetilde{R}>\widehat{R}.
$$
\end{lem}
\begin{proof}
By the given condition, we immediately have
$$
\|w_{j,R} -w(\p_j)\|_\infty \to 0 	\quad \text{for each} \quad j\in\{1,\dots,N\}
$$
as $R\to 0^+$. Let
$$
f_{j,R}(\p)=w_{j,R}(\p) -w(\p_j)=\widetilde{\omega}\Big(\frac{d(\p,\p_j)}{2R}\Big) \Big(w(\p)-w(\p_j) \Big).
$$
Then for $\p,\q \in \B_\M(\p_j,2R)$ with $\p\neq \q$, it follows from \eqref{w-BC} that
\begin{eqnarray*}
\lefteqn{\frac{|f_{j,R}(\p)- f_{j,R}(\q)|}{d(\p,\q)^\alpha}} \\
&\leq & \frac{\Big| \widetilde{\omega}\Big(\frac{d(\p,\p_j)}{2R}\Big) - \widetilde{\omega}\Big(\frac{d(\q,\p_j)}{2R}\Big) \Big|}{d(\p,\q)^\alpha} | w(\p)-w(\p_j)| + \Big|\widetilde{\omega}\Big(\frac{d(\q,\p_j)}{2R}\Big) \Big| \frac{|w(\p)- w(\q)|}{d(\p,\q)^\alpha} \\
&\leq & \frac{\Big| \widetilde{\omega}\Big(\frac{d(\p,\p_j)}{2R}\Big) - \widetilde{\omega}\Big(\frac{d(\q,\p_j)}{2R}\Big) \Big|}{ \Big|\frac{d(\p,\p_j)}{2R} -\frac{d(\q,\p_j)}{2R} \Big|^\alpha}\frac{\Big|\frac{d(\p,\p_j)}{2R} -\frac{d(\q,\p_j)}{2R} \Big|^\alpha}{d(\p,\q)^\alpha} | w(\p)-w(\p_j)| + C R^{r-\alpha} \\
&\leq & C R^{r-\alpha},
\end{eqnarray*}
for certain $C>0$. If $\M=(\R^n, g_n)$, then let
$$
g_{0,\widetilde{R}}(\p)=w_{0,\widetilde{R}}(\p)-w_{\infty}=\Big(1-\widetilde{\omega}\Big(\frac{d(\p,0)}{2\widetilde{R}}\Big)\Big) (w(\p)-w_{\infty}).
$$
For any $\p,\q \in \R^n\backslash \B_\M(0,\widetilde{R})$ with $\p\neq \q$, we have
\begin{eqnarray}\nonumber
\frac{|g_{0,\widetilde{R}}(\p)- g_{0,\widetilde{R}}(\q)|}{d(\p,\q)^\alpha} \hspace{-2mm}
&\leq& \hspace{-2mm} |w(\q)-w_{\infty}|\frac{\Big|\widetilde{\omega}\Big(\frac{d(\p,0)}{2\widetilde{R}}\Big)-\widetilde{\omega}\Big(\frac{d(\q,0)}{2\widetilde{R}}\Big)\Big|}{ \Big|\frac{d(\p,0)}{2\widetilde{R}} -\frac{d(\q,0)}{2\widetilde{R}} \Big|^\alpha}\frac{ \Big|\frac{d(\p,0)}{2\widetilde{R}} -\frac{d(\q,0)}{2\widetilde{R}} \Big|^\alpha}{d(\p,\q)^{\alpha}}\\\label{rhsft}
&&+\Big|1-\widetilde{\omega}\Big(\frac{d(\p,0)}{2\widetilde{R}}\Big)\Big|\frac{|w(\p)-w(\q)|}{d(\p,\q)^{\alpha}}.
\end{eqnarray}
Due to \eqref{winfinity}, the first term on the right hand side of \eqref{rhsft} is $\leq C \widetilde{R}^{-\alpha}$. For the second term on the right hand side of \eqref{rhsft}, if $d(\p,\q)\leq\varepsilon_{0}$, for some $\varepsilon_{0}>0$, then it is $\leq C \varepsilon_{0}^{r-\alpha}<\varepsilon/2$, by taking $\varepsilon_{0}$ small enough. If $d(\p,\q)\geq\varepsilon_{0}$, then 
$$
\frac{|g_{0,\widetilde{R}}(\p)- g_{0,\widetilde{R}}(\q)|}{d(\p,\q)^\alpha} \leq C\varepsilon_{0}^{-\alpha}(|w(\p)-w_{\infty}|+|w_{\infty}-w(\q)|)\leq\varepsilon/2,
$$
by choosing $\widetilde{R}$ sufficiently large due to \eqref{winfinity}. The result follows by writing
$$
\M\times \M=\{(\p,\q)\in \R^n\times \R^n\, |\, d(\p,\q)\leq\varepsilon_{0}\}\cup \{(\p,\q)\in \R^n\times \R^n\, |\, d(\p,\q)\geq\varepsilon_{0}\}.
$$
\end{proof}

\begin{lem}\label{Lem: commutator smooth fn}
If $\phi\in BC^\infty(\M)$, then for any $s\geq 0$ and $c>0$ 
$$
[\phi,(c+\sL)^{\sigma}]\in \L(H^{s+2\sigma}_p(\M,V),H^{s+1-\varepsilon}_p(\M,V))
$$ 
for any $\varepsilon>0$. Moreover
$$
\|[\phi,(c+\sL)^{\sigma}]\|_{\L(H^{s+2\sigma}_p(\M,V),H^{s+1-\varepsilon}_p(\M,V))} \leq M=M(c_0),\quad c>c_0,
$$
for any fixed $c_0>0$.
\end{lem}
\begin{proof}
Let $A =c+\sL.$ Then, for any $u\in H^{s+2\sigma}_p(\M,V)$, \eqref{S4: fractional Delta def} implies
\begin{eqnarray*}\nonumber
\lefteqn{(\phi A^{\sigma}-A^{\sigma}\phi)u}\\\nonumber
&=&\frac{\sin(\pi \sigma)}{\pi}\int_{0}^{+\infty}x^{\sigma-1}\Big\{ [ \phi, A](A+x)^{-1}+A [\phi,(A+x)^{-1}] \Big\} u\, dx\\\nonumber
&=& -\frac{\sin(\pi \sigma)}{\pi}\int_{0}^{+\infty}x^\sigma (A +x)^{-1}[A,\phi ] (A +x)^{-1} u \, dx.
\end{eqnarray*}
Note that $[A,\phi ]$ is a first order differential operator.
When $x>1$, \cite[Lemma 2.3.3]{Tan} implies
\begin{eqnarray*}
\lefteqn{\| x^\sigma (A +x)^{-1}[A,\phi ] (A +x)^{-1} u \|_{s+1-\varepsilon,p} }\\
&\leq & x^{\sigma-1} \|[A,\phi ] A^{-\frac{1}{2}-\frac{\varepsilon}{4}} A^{1-\sigma-\frac{\varepsilon}{4}}(A +x)^{-1} A^{ \sigma +\frac{\varepsilon}{2} -\frac{1}{2}} u \|_{s+1-\varepsilon,p}\\
&\leq & C x^{-1-\frac{\varepsilon}{4}} \|u \|_{s+2\sigma,p}
\end{eqnarray*}
for $\varepsilon>0$ sufficiently small;
and when $x\leq 1$, letting $B=\frac{c}{2}+\sL$, we have
\begin{eqnarray*}
\lefteqn{\| x^\sigma (A +x)^{-1}[A,\phi ] (A +x)^{-1} u \|_{s+1-\varepsilon,p}} \\
&\leq & x^{\sigma-1} \|[A,\phi ] B^{-\frac{1}{2}-\frac{\varepsilon}{4}} B^{1-\sigma-\frac{\varepsilon}{4}}(A +x)^{-1} B^{ \sigma +\frac{\varepsilon}{2} -\frac{1}{2}} u \|_{s+1-\varepsilon,p}\\
&\leq & C \left(\frac{c}{2}+x \right)^{-1-\frac{\varepsilon}{4}} \|u \|_{s+2\sigma,p}.
\end{eqnarray*} 
These two estimates establish the assertion.
\end{proof}

\begin{prop}\label{Prop: strong reg s<1}
Suppose that $w$ satisfies \eqref{w-BC}-\eqref{positive} with $r\in (0,1]$ and let $f\in L_q(J,H^s_p(\M,V))$ for some $s\in [0,r)$. Then there exists a unique
$$
u\in L_q(J, H^{s+2\sigma}_p(\M,V)) \cap H^1_q(J, H^s_p(\M,V)
$$
solving \eqref{MR-eq}.
\end{prop}
\begin{proof}
The result follows by similar steps as in the proof of \cite[Theorem~6.2]{RoidosShao}, where we have to take Lemma~\ref{Lem: commutator smooth fn} into account. More precisely, if $w_{j,R}$, $j\in\{1,\dots,N\}$, $R>0$, are as in \eqref{localizerfunt}-\eqref{localizerfuntforRm}, let 
\begin{equation}\label{eqpertr}
w_{j,R}A^{\sigma}=w(\p_j)A^{\sigma}+ (w_{j,R}-w(\p_j))A^{\sigma}:H^{s+2\sigma}_p(\M,V) \rightarrow H^{s}_p(\M,V),
\end{equation}
where $A =c_{0}+\sL$, $c_{0}>0$; and in the case of $\M=(\R^n, g_n)$, we define
\begin{equation}\label{eqpertr2}
w_{0,\widetilde{R}}A^{\sigma}=w_{\infty}A^{\sigma}+ (w_{0,\widetilde{R}}-w_{\infty})A^{\sigma}:H^{s+2\sigma}_p(\M,V) \rightarrow H^{s}_p(\M,V).
\end{equation}
Note that, by Proposition~\ref{embBCH}~(ii), elements in $BC^r(\M)$ act by multiplication as bounded maps on $H^s_p(\M,V)$. Therefore, for every $\theta\in(\pi/2,\pi)$ and every $c>0$, by Lemma~\ref{Lem: localization est}, \eqref{eqpertr} and \cite[Theorem 1]{KuWe}, after choosing $R$ sufficiently small and $\widetilde{R},N$ large enough, both operators \eqref{eqpertr}-\eqref{eqpertr2} belong to $\mathcal{R}(\theta)$. As a consequence, due to standard sectoriality of $w_{j,R}A^{\sigma}$, $ j\in\{1,\dots,N\}$, $w_{0,\widetilde{R}}A^{\sigma}$ and $A^{\sigma}$, \cite[(6.45)]{RoidosShao} holds true. 

Set $J=\{1,\dots,N\}$ when $(\M,g)$ is a closed manifold or $J=\{0,1,\dots,N\}$ when $(\M,g)=(\R^n,g_n)$.
Moreover, we put $R_j=R$ when $j\in \{1,\dots,N\}$ and $R_0=\widetilde{R}$. Choose $\phi_{j}\in BC^{\infty}(\M)$, $j\in J$, to be a partition of unity subordinated to the cover $\{U_{j}\}_{j \in J}$. Moreover, let $\psi_{j}\in BC^{\infty}(\M)$, $j\in J$, supported on $U_{j}$, taking values on $[0,1]$ and satisfying $\psi_{j}\equiv 1$ on the support of $\phi_{j}$. Then by Lemma \ref{Lem: commutator smooth fn}, similarly to \cite[(6.47)]{RoidosShao}, for sufficiently large $c$, we can construct a left inverse $L(\lambda)$ of $wA^{\sigma}+c+\lambda$, $\lambda\in \Sigma_{\theta}$, that belongs to the space $\mathcal{L}(H^{s}_p(\M,V),H^{s+2\sigma}_p(\M,V))$. More precisely, we have that 
$$
L(\lambda)=\sum_{k=0}^{\infty}Q^{k}(\lambda)R(\lambda), \quad \lambda\in \Sigma_{\theta},
$$
where 
$$
Q(\lambda)=\sum_{j\in J} \psi_{j}(w_{j,R_j}A^{\sigma}+c+\lambda)^{-1}w_{j,R_j}[A^{\sigma},\phi_{j}]
$$
and
$$ 
R(\lambda)=\sum_{j\in J} \psi_{j}(w_{j,R_j}A^{\sigma}+c+\lambda)^{-1}\phi_{j}.
$$
Furthermore, similarly to \cite[(6.48)]{RoidosShao} we can show that $L(\lambda)$ is also a right inverse of $wA^{\sigma}+c+\lambda$, $\lambda\in \Sigma_{\theta}$. After having the above expression of the resolvent of $wA^{\sigma}+c$, we can show $R$-sectoriality of angle $\theta$ for this operator as in the proof of \cite[Theorem~6.2]{RoidosShao}, i.e. similarly to \cite[(6.49)]{RoidosShao} and the estimates below. Next, $R$-sectoriality for $w \sL^{\sigma}+c$ for large $c$ is obtained by the Step 2 of the proof of \cite[Theorem~6.2]{RoidosShao}. Then the result follows by Theorem \ref{KaWeTh}.
\end{proof}

\begin{remark}
The proof of Proposition \ref{Prop: strong reg s<1} is based on the generalization of freezing-of-coefficients method to the case of non-local operator of certain type. Such an extension was first demonstrated in the proof of \cite[Theorem~6.2]{RoidosShao}. One of the main ingredients of the proof is the observation that the commutator of the fractional powers of the Laplacian and a function in the class $BC^\infty(\M)$ is indeed of lower order in a sectoriality sense, see Lemma \ref{Lem: commutator smooth fn}. Moreover, instead of using an $\varepsilon-C_{\varepsilon}$ argument as in the classical case (i.e. the case of differential operators in $L^{p}$-spaces, see, e.g. the proof of \cite[Theorem 5.7]{DenHiePru03}), similarly to the proof of \cite[Theorem~6.2]{RoidosShao}, we proceed by using the decay properties of the resolvent of a sectorial operator, i.e. \cite[Lemma 2.3.3]{Tan}, in order to construct a left and right inverse for $wA^{\sigma}+c+\lambda$, $\lambda\in \Sigma_{\theta}$. 
\end{remark}

With a little abuse of notation, we denote
$$ 
F:=[\nabla, \sL_s ]=: \nabla \sL^\eta_{s,\tau} - \sL^\eta_{s-1,\tau+1} \nabla,
$$ 
which is a second order differential operator. 
Then, we have 
\begin{align*}
[\nabla, (\lambda+\sL_s )^{-1}] = - (\lambda+\sL^\eta_{s,\tau+1} )^{-1} F (\lambda+\sL^\eta_{s,\tau} )^{-1} 
\end{align*}
for all $\lambda\in \Sigma_\theta$, where $\theta> \pi/2$ is the sectorial angle of $\sL_s$ asserted in Theorem~\ref{Thm: sectoriality Lap}. Given any $\delta>0$ and $u\in H^{s+2\sigma + \delta}_p(\M,V)$,
since
\begin{eqnarray*}
\lefteqn{\nabla \sL (\lambda +\sL )^{-1}u \,\,\, = \,\,\, \nabla u - \lambda\nabla (\lambda +\sL )^{-1}u} \\
&=& \nabla u - \lambda (\lambda +\sL )^{-1} \nabla u - \lambda [\nabla, (\lambda +\sL_s )^{-1}] u \\
&=& \sL (\lambda +\sL )^{-1} \nabla u - \lambda (\lambda+\sL^\eta_{s,\tau+1} )^{-1} F (\lambda+\sL^\eta_{s,\tau} )^{-1} u,
\end{eqnarray*}
we have
\begin{eqnarray*}
\lefteqn{ \nabla \sL^\sigma u 
\,\,\, = \,\,\, \frac{\sin(\pi\sigma)}{\pi} \int_0^\infty x^{\sigma-1} \nabla \sL (x +\sL )^{-1}u \, dx} \\
&=& \sL^\sigma \nabla u 
 -\frac{\sin(\pi\sigma)}{\pi} \underbrace{\int_0^\infty x^\sigma (x+\sL^\eta_{s,\tau+1} )^{-1} F (x+\sL^\eta_{s,\tau} )^{-1} u\, dx}_{(\ast)} .
\end{eqnarray*}
 
To estimate ($\ast$), we first note that for $x>1$ 
\begin{eqnarray*}
\lefteqn{\| (x+\sL^\eta_{s,\tau+1} )^{-1} F (x+ \sL^\eta_{s,\tau} )^{-1} u\|_{s,p}} \\
&\leq & \frac{M}{x} \| F (x+\sL )^{-1} u\|_{s,p} \\
&\leq & \frac{M}{x} \| (x+\sL )^{-1} u\|_{s+2,p} \\
&\leq & \frac{M}{x} \Big[ \|\sL (x+\sL )^{-1} u\|_{s,p} + \|(x+\sL )^{-1} u\|_{s,p} \Big] \\
&\leq & \frac{M}{x} \Big[\frac{C}{x^{\sigma+\varepsilon}}\|u\|_{s+2\sigma+2\varepsilon,p} + \frac{C}{x}\|u\|_{s,p} \Big]
\end{eqnarray*}
for $\varepsilon>0$ sufficiently small. 
The last step follows from \cite[Lemma~2.3.3]{Tan}.	
When $x\leq 1$, we will use the following lemma.
\begin{lem}
\label{Appendix: intermediate-resolvent}
Let $X_1 \xhookrightarrow{d} X_0$ be a pair of Banach spaces, where $X_j$ is equipped with norm $\|\cdot\|_j$. 
Suppose that $\vartheta\in (\pi/2,\pi)$, and $A\in \S(\vartheta)$ with domain $\dom(A) =X_1$. 
Let $(\alpha,p),(\beta,p)\in \{(0,1)\times[1,\infty]\}\cup \{(1,\infty)\}$ with $\beta\geq \alpha$.
There exists $C=C(p,\alpha,\beta)>0$ such that for all $t\in (0,1]$
$$
\| (t-A)^{-1} \|_{\L((X_{0},X_{1})_{\alpha,p},(X_{0},X_{1})_{\beta,p})}\leq C t^{\beta-1-\alpha} .
$$
\end{lem}
\begin{proof}
It follows from \cite[Propositions~2.2.2 and 2.2.9]{Lunar95} that 	
\begin{equation}
\label{eq: intermediae space}
\|t^{\beta-\alpha}e^{tA}\|_{\L((X_{0},X_{1})_{\alpha,p},(X_{0},X_{1})_{\beta,p})} \leq C=C(p,\alpha,\beta).
\end{equation}
Since 
$$
(t+A)^{-1}=\int_0^\infty e^{-ts} e^{-sA}\, ds ,
$$
we can compute for all $u\in (X_{0},X_{1})_{\alpha,p}$
\begin{eqnarray*}
\| (t+A)^{-1} u\|_{(X_{0},X_{1})_{\beta,p}} & \leq& \int_0^\infty e^{-ts} \|e^{-sA} u \|_{(X_{0},X_{1})_{\beta,p}}\, ds \\
&\leq & \, C \|u \|_{(X_{0},X_{1})_{\alpha,p}} \int_0^\infty e^{-ts} s^{\alpha-\beta}\, ds \\
&\leq & \, C \|u \|_{(X_{0},X_{1})_{\alpha,p}} t^{\beta-1-\alpha} \int_0^\infty e^{-s} s^{\alpha-\beta}\, ds .
\end{eqnarray*}
\end{proof}

By Lemma~\ref{Appendix: intermediate-resolvent}, when $x\leq 1,$ we have an even better estimate
\begin{eqnarray*}
 \| (x+\sL^\eta_{s,\tau+1} )^{-1} F (x+ \sL^\eta_{s,\tau} )^{-1} u\|_{s,p} &\leq &\frac{M}{x} \|(x+\sL )^{-1} u\|_{s+2,p} \\
&\leq & \frac{M}{x^{1+\sigma- \varepsilon}} \|u\|_{s+2\sigma,p}.
\end{eqnarray*}
Therefore,
\begin{align}
\label{commutator est - grad}
\|[\nabla ,\sL_s^\sigma] u \|_{s,p}
= :\| \nabla (\sL^\eta_{s,\tau})^\sigma u - (\sL^\eta_{s-1,\tau+1})^\sigma (\nabla u)\|_{s,p} 
\leq M \|u\|_{s+ 2\sigma+2\varepsilon,p}.
\end{align}

\begin{theorem}\label{main theorem}
Assume that $w\in BC^r(\M)$ satisfies \eqref{w-BC}-\eqref{positive} and let 
\begin{center}
$f\in L_{q}(J, H^s_p(\M))$ for some $p,q\in(1,\infty)$ and $s\in [0,r)$. 
\end{center}
Then the solution to \eqref{MR-eq} satisfies
\begin{equation}
\label{eq: higher oder MR}
u\in H^1_q(J, H^s_p(\M, V)) \cap L_q(J, H^{s+2\sigma}_p(\M,V)), 
\end{equation}
i.e.
$$
w \sL^\sigma \in \mathcal{MR}_q(H^{s+2\sigma}_p(\M,V), H^s_p(\M, V)) .
$$
\end{theorem}
\begin{proof}
When $s\in [0,1)$, the assertion is already proved. Consider the case $s\in [1,2)$ and $r>1$. Choose $\varepsilon>0$ so small that $s-1+2\varepsilon<1$.
Taking $\nabla $ on both sides of \eqref{MR-eq} yields
$$
\partial_t v + w (\sL^\eta_{s-1,\tau+1})^\sigma v = \nabla f - \nabla w \otimes(\sL^\eta_{s ,\tau })^\sigma u - w[\nabla, (\sL_s)^\sigma]u,
$$
where $v=\nabla u$. 
By Proposition~\ref{Prop: strong reg s<1}, we already know that
$$
u \in L_q(J, H^{s-1+2\sigma+2\varepsilon}_p(\M,V))\cap H^1_q(J,H^{s-1+2\varepsilon}_p(\M,V)).
$$
The standard pointwise multiplication theory, cf. \cite[Theorem~9.2]{Ama13}, implies
$$
\nabla w \otimes (\sL^\eta_{s ,\tau })^\sigma u \in L_q(J,H^{s-1}_p(\M,V^\eta_{\tau+1}));
$$
\eqref{commutator est - grad} gives 
$$
[\nabla, (\sL_s)^\sigma]u \in L_q(J,H^{s-1}_p(\M,V^\eta_{\tau+1})).
$$
Note that in the above step, we need $u \in L_q(J, H^{s-1+2\sigma+2\varepsilon}_p(\M,V))$ in view of \eqref{commutator est - grad}.
It follows from Proposition~\ref{Prop: strong reg s<1} that
$$
v \in L_q(J, H^{s-1+2\sigma}_p(\M,V^\eta_{\tau+1}))\cap H^1_q(J,H^{s-1}_p(\M,V^\eta_{\tau+1})).
$$
This proves \eqref{eq: higher oder MR} for $s\in [1,2)$. The general case follows by induction.
\end{proof}
 
\section{Applications}\label{Section 6}

In this section, we will apply Theorem~\ref{main theorem} and the following theorem by P. Cl\'ement and S. Li to study two quasilinear parabolic equations. 
\begin{theorem}[{\rm Cl\'ement and Li, \cite[Theorem 2.1]{CL}}]\label{ClementLi}
Suppose that $X_{1}\overset{d}{\hookrightarrow}X_{0}$ is a continuously and densely injected complex Banach couple.
Let $U$ be an open subset of $(X_0,X_1)_{1-\frac{1}{q},q}$, where $q\in(1,\infty)$. Consider the problem
\begin{equation}
\label{aqpp1}
\left\{\begin{aligned}
u'(t)+A(u(t))u(t)&= F(t, u(t )) , &&t\geq 0\\
u(0)&= u_{0}, &&
\end{aligned}\right. 
\end{equation}
where $u_{0}\in U$. 
Assume that:\\
{\em (H1)} $A \in C^{1-}(U,\mathcal{L}(X_{1},X_{0}))$.\\
{\em (H2)} $F \in C^{1-,1- }([0,T_0]\times U, X_{0})$ for some $T_0>0$.\\
{\em (H3)} $A(u_{0})\in \mathcal{MR}_q(X_1,X_0)$.\\
Then, there exists a $T\in (0,T_0]$ and a unique 
$$
u\in H_q^{1}((0,T),X_{0})\cap L_q((0,T),X_{1})
$$ 
solving \eqref{aqpp1}. 
\end{theorem}
 
\subsection{Fractional porous medium equation}\label{Section 6.1}

As an application, we consider first the following fractional porous medium equation (FPME)
\begin{equation}\label{FPMEQ}
\left\{\begin{aligned}
\partial_t u +(-\Delta )^\sigma (|u|^{m-1}u )&=f &&\text{on}&&\M\times (0,\infty);\\
u(0)&=u_0 &&\text{on}&&\M,
\end{aligned}\right.
\end{equation}
where $(\M,g)$ is an $n$-dimensional closed manifold, $\sigma\in (0,1)$ and 
\begin{equation}\label{fregH}
f\in C([0,T_{0}], H^s_p(\M))
\end{equation}
for some $T_{0}>0$, $s\geq0$ and $p\in(1,\infty)$.
Further, $\Delta=-\nabla^* \circ \nabla$ is the Laplace-Beltrami operator, cf. \eqref{exp of L}.

\begin{theorem}[Smoothing for the FPME]\label{Thsmoothfpme}
Let $u_{0}\in B^{s+2\sigma-2\sigma/q}_{p,q}(\M)$ for some $q\in (1,\infty)$, where $p$ and $s$ are as in \eqref{fregH}. Assume that $2\sigma>2\sigma/q+n/p$ and $u_{0}>c$ on $\M$, for certain $c>0$. Then, there exists a $T\in (0,T_0]$ and a unique 
\begin{eqnarray}\label{ureg1st}
\lefteqn{u\in L_q((0,T), H^{s+2\sigma}_p(\M)) \cap H^1_q( (0,T), H^s_p(\M))}\\\label{ureg2nd}
&&\quad \quad \quad \quad \quad \hookrightarrow C([0,T], B^{s+2\sigma-2\sigma/q}_{p,q}(\M) )
\end{eqnarray}
solving \eqref{FPMEQ}. If, in addition,
$$
f\in \bigcap _{\nu>0}L_{q}( (0,T_{0}), H^{\nu}_{p}(\M))\cap C^{\nu}((0,T_{0}),H^{\nu}_{p}(\M)),
$$ 
then $u$ satisfies the regularity
\begin{eqnarray}\label{regFPME2}
u\in \bigcap _{\nu>0} C^{\nu}((0,T),H^{\nu}_{p}(\M)).
\end{eqnarray}
\end{theorem}
\begin{proof}

We consider first the problem 
\begin{equation}\label{WFPMEQ}
\left\{\begin{aligned}
\partial_t w +mw^{\frac{m-1}{m}}(-\Delta)^{\sigma}w&=mw^{\frac{m-1}{m}}f &&\text{on}&&\M\times (0,\infty);\\
w(0)&=u_{0}^{m} &&\text{on}&&\M.
\end{aligned}\right.
\end{equation}
Concerning \eqref{ureg1st}-\eqref{ureg2nd}, we will apply the theorem of P. Cl\'ement and S. Li, i.e. Theorem \ref{ClementLi}, to the above equation and then we will recover the required existence and regularity result for the original problem. Define the Banach couple $X_{0}=H^s_p(\M)$, $X_{1}=H^{s+2\sigma}_p(\M)$, the operator family $A(\cdot)=A_{s}(\cdot)=m(\cdot)^{\frac{m-1}{m}}(-\Delta_{s})^{\sigma}$, where $\Delta_{s}$ denotes the map $\Delta:H^{s+2}_p(\M)\rightarrow H^{s}_p(\M)$, and let the potential term $F(\cdot)=m(\cdot)^{\frac{m-1}{m}}f$.
Note that Proposition \ref{Prop: interpolation} implies
$$
B^{s+2\sigma-2\sigma/q}_{p,q}(\M) \doteq (X_0,X_1)_{1-1/q,q}.
$$ 
By Proposition \ref{embBCH}, we have
\begin{equation}\label{interemb}
u_{0}^{\alpha}\in B^{s+2\sigma-2\sigma/q}_{p,q}(\M) \hookrightarrow H^{\xi}_p(\M) \hookrightarrow BC^{r}(\M),
\end{equation}
for any $\alpha\in\mathbb{R}$ and $s+n/p<r+n/p<\xi<s+2\sigma-2\sigma/q$. By the relation 
$$
|v-mu_{0}^{m-1}|\leq \|v-mu_{0}^{m-1}\|_\infty \leq C_{1}\|v-mu_{0}^{m-1}\|_{B^{s+2\sigma-2\sigma/q}_{p,q} },
$$
valid for certain $C_{1}>0$, we choose an open ball $U$ in $B^{s+2\sigma-2\sigma/q}_{p,q}(\M) $ around $mu_{0}^{m-1}$ of sufficiently small radius, such that 
\begin{equation}\label{lowerbound}
\mathrm{Re}(v)\geq c/2 \quad \text{for each} \quad v\in U. 
\end{equation}
Let $\Gamma$ be a finite positively oriented simple path in $\{z\in\mathbb{C}\, |\, \mathrm{Re}(z)>0\}$ that surrounds $\{\mathrm{Ran}(v)\, |\, v\in U\}$. For each $v_{1},v_{2}\in U$ we have
\begin{equation}\label{ualpha}
v_{1}^{\alpha}-v_{2}^{\alpha}=\frac{1}{2\pi i}\int_{\Gamma}\Big(\frac{\lambda^{\alpha}}{\lambda-v_{1}}-\frac{\lambda^{\alpha}}{\lambda-v_{2}}\Big)d\lambda=\frac{v_{1}-v_{2}}{2\pi i}\int_{\Gamma}\frac{\lambda^{\alpha}}{(\lambda-v_{1})(\lambda-v_{2})}d\lambda.
\end{equation}
Hence, by Proposition \ref{embBCH}~(ii) 
\begin{eqnarray*}
\lefteqn{ \|A_{s}(v_{1})-A_{s}(v_{2})\|_{\mathcal{L}(H^{s+2\sigma}_p(\M),H_{p}^{s}(\M))}}\\
&=&\|(v_{1}^{\frac{m-1}{m}}-v_{2}^{\frac{m-1}{m}})(-\Delta_s )^{\sigma}\|_{\mathcal{L}(H^{s+2\sigma}_p(\M),H_{p}^{s}(\M))}\\
&\leq& C_{2} \|(v_{1}^{\frac{m-1}{m}}-v_{2}^{\frac{m-1}{m}})\cdot\|_{\mathcal{L}(H_{p}^{s}(\M))}\\
&\leq& C_{3} \|v_{1}^{\frac{m-1}{m}}-v_{2}^{\frac{m-1}{m}}\|_{H^{\xi}_p(\M)}\leq C_{4} \|v_{1}^{\frac{m-1}{m}}-v_{2}^{\frac{m-1}{m}}\|_{ B^{s+2\sigma-2\sigma/q}_{p,q} }
\end{eqnarray*}
for certain $C_{2}, C_{3}, C_{4}>0$, so that
\begin{equation}\label{aLipcont}
 \|A_{s}(v_{1})-A_{s}(v_{2})\|_{\mathcal{L}(H^{s+2\sigma}_p(\M),H^s_p(\M))}\leq C_{5}\|v_{1}-v_{2}\|_{ B^{s+2\sigma-2\sigma/q}_{p,q} }
\end{equation}
for some $C_{5}>0$ due to \eqref{interemb} and \eqref{ualpha}.

Furthermore, for each $t_{1},t_{2}\in [0,T_{0}]$, by Proposition~\ref{embBCH}~(ii), we have
\begin{eqnarray}\nonumber
\lefteqn{\|F(v_{1},t_{1})-F(v_{2},t_{2})\|_{s,p}}\\\nonumber
&=&m\|(v_{1}^{\frac{m-1}{m}}-v_{2}^{\frac{m-1}{m}})f(t_{1})+v_{2}^{\frac{m-1}{m}}(f(t_{1})-f(t_{2}))\|_{s,p}\\\nonumber
&\leq& C_{6}\Big(\|v_{1}^{\frac{m-1}{m}}-v_{2}^{\frac{m-1}{m}}\|_{H^{\xi}_p(\M)}\|f\|_{C([0,T_{0}],H^s_p(\M))}\\\nonumber
&&+\|v_{2}^{\frac{m-1}{m}}\|_{H^{\xi}_p(\M)}\|f(t_{1})-f(t_{2})\|_{s,p}\Big) \\\nonumber
&\leq& C_{7}\Big(\|v_{1}^{\frac{m-1}{m}}-v_{2}^{\frac{m-1}{m}}\|_{ B^{s+2\sigma-2\sigma/q}_{p,q}}\|f\|_{C([0,T_{0}],H^s_p(\M))}\\\nonumber
&&+\|v_{2}^{\frac{m-1}{m}}\|_{ B^{s+2\sigma-2\sigma/q}_{p,q}}\|f(t_{1})-f(t_{2})\|_{s,p}\Big) \\\label{Flipcont}
&\leq& C_{8}\Big(\|v_{1}-v_{2}\|_{B^{s+2\sigma-2\sigma/q}_{p,q}}+|t_{1}-t_{2}|\Big)
\end{eqnarray}
for some $C_{6},C_{7}, C_{8}>0$, where we have used \eqref{interemb} and \eqref{ualpha} once more.

Clearly, $A_{s}(u_{0}) $ has maximal $L_q$-regularity due to Theorem \ref{main theorem} and \eqref{interemb}. By Theorem \ref{ClementLi}, there exists a $T\in(0,T_{0}]$ and a unique 
\begin{equation}\label{wreg1}
 w\in H^1_q((0,T), H^s_p(\M)) \cap L_q((0,T), H^{s+2\sigma}_p(\M)) 
\end{equation}
solving \eqref{WFPMEQ}. In addition, due to \eqref{inembmaxreg}, we also have
\begin{equation}\label{wregA}
w\in C([0,T], B^{s+2\sigma-2\sigma/q}_{p,q}(\M)).
\end{equation}
Hence, by choosing $T>0$ small enough we can make $w(t)\in U$ for each $t\in[0,T)$. Then, due to Proposition~\ref{embBCH}~(iii) and \eqref{interemb}, for any $\alpha\in \mathbb{R}$
\begin{equation}\label{wreg2}
w^{\alpha}\in C([0,T], B^{s+2\sigma-2\sigma/q}_{p,q}(\M)).
\end{equation}
 
By the relation 
$$
\partial_{t}w^{\frac{1}{m}}=m^{-1}w^{\frac{1-m}{m}}\partial_{t}w,
$$
\eqref{wreg1}, \eqref{wreg2} and Proposition~\ref{embBCH} we deduce that $w^{1/m}\in H^1_q((0,T), H^s_p(\M))$. Furthermore, due to the Banach algebra property of $H^{s+2\sigma}_p(\M)$ and \eqref{wreg1}, we see that $w^{1/m}(t)\in H^{s+2\sigma}_p(\M)$ for almost all $t\in[0,T]$, so that the function $u=w^{1/m}$ satisfies $\partial_t u +(-\Delta )^\sigma (u^{m})=f $ for almost all $t\in[0,T]$. We estimate
\begin{eqnarray*}
\lefteqn{\|u(t)\|_{ s+2\sigma,p} \leq C_{9}(\|u(t)\|_{s,p}+\|(-\Delta )^\sigma u(t)\|_{s,p})}\\
&\leq& C_{9}(\|u(t)\|_{s,p}+\|\partial_{t}u(t)\|_{s,p}+\|f(t)\|_{s,p}),
\end{eqnarray*}
for certain $C_{9}>0$ and almost all $t\in[0,T]$. By integrating the above inequality over $t\in[0,T]$, we obtain \eqref{ureg1st}. Then, \eqref{ureg2nd} follows by \eqref{inembmaxreg}.

Concerning \eqref{regFPME2}, we will apply the smoothing result \cite[Theorem 3.1]{RoSch18} to \eqref{WFPMEQ} and then we will recover again the required regularity for $u$. Hence, we examine the conditions (i), (ii) and (iii) of \cite[Theorem 3.1]{RoSch18}. We choose the Banach scales 
$$
Y_{0}^{j}=H^{s+jb}_{p}(\M), \quad Y_{1}^{j}=H^{s+2\sigma+jb}_{p}(\M), \quad j\in\Nz,
$$
where $b\in(0,2\sigma-\frac{2\sigma}{q}-\frac{n}{p})$ is fixed. Moreover, choose $A(\cdot)$, $F$ as before and let 
$Z=\{v\in U\, |\, \mathrm{Im}(v)=0\}$.

Condition (i). By the previous step, we have the existence of $w$ as in \eqref{wreg1} satisfying $w(t)\in Z$ for all $t\in[0,T]$; here we have taken the complex conjugate to \eqref{WFPMEQ} and then used the above uniqueness result, i.e. we have obtained in addition that $\mathrm{Im}(w(t))=0$, $t\in[0,T]$. By this observation, \eqref{interemb}, \eqref{lowerbound}, \eqref{wregA} and Theorem \ref{main theorem}, we also have that, for each $t\in[0,T]$, the operator $A(w(t))\in \mathcal{MR}_{q}(Y_{1}^{0},Y_{0}^{0})$. Finally, due to \eqref{aLipcont} and \eqref{wregA}, we deduce that $A(w(\cdot))\in C([0,T],\mathcal{L}(Y_{1}^{0},Y_{0}^{0}))$.

Condition (ii). Let $h\in Z\cap (Y_{0}^{j},Y_{1}^{j})_{1-\frac{1}{q},q}$, $j\in\mathbb{N}$. By Proposition \ref{embBCH}~(i) and \eqref{interemb}, we have
$$
h\in B^{s+jb+2\sigma-\frac{2\sigma}{q}}_{p,q}(\M) \hookrightarrow H^{\xi_{j}}_p(\M) \hookrightarrow BC^{r_{j}}(\M),
$$
where $s+(j+1)b + n/p<r_{j} + n/p<\xi_{j}<s+jb+2\sigma-\frac{2\sigma}{q}$. Thus, due to \eqref{lowerbound} and Theorem \ref{main theorem} we obtain that $A(h)\in \mathcal{MR}_q(Y_{1}^{j+1}, Y_{0}^{j+1})$. Now let $\eta \in C([0,T],Z\cap (Y_{0}^{j},Y_{1}^{j})_{1-\frac{1}{q},q})$. Similarly to \eqref{aLipcont}, by \eqref{lowerbound} and Proposition~\ref{embBCH}, we get 
\begin{eqnarray*}
 &&\hspace{-60pt}\|A_{s+(j+1)b}(\eta(t_{1}))-A_{s+(j+1)b}(\eta(t_{2}))\|_{\mathcal{L}(H^{s+(j+1)b+2\sigma}_p(\M),H^{s+(j+1)b}_p(\M))}\\
 &\leq& C_{10} \|(\eta^{\frac{m-1}{m}}(t_{1})-\eta^{\frac{m-1}{m}}(t_{2}))\cdot\|_{\mathcal{L}(H^{s+(j+1)b}_p(\M))}\\
 &\leq& C_{11} \|\eta^{\frac{m-1}{m}}(t_{1})-\eta^{\frac{m-1}{m}}(t_{2})\|_{H^{\xi_{j}}_p(\M)}\\
 &\leq& C_{12} \|\eta^{\frac{m-1}{m}}(t_{1})-\eta^{\frac{m-1}{m}}(t_{2})\|_{ B^{s+jb+2\sigma-\frac{2\sigma}{q}}_{p,q}}\\
 & \leq& C_{13}\|\eta(t_{1})-\eta(t_{2})\|_{B^{s+jb+2\sigma-\frac{2\sigma}{q}}_{p,q}}
 \end{eqnarray*}
for some $C_{10},C_{11},C_{12}, C_{13}>0$, where $t_{1},t_{2}\in [0,T]$. This implies that 
$$
A_{s+(j+1)b}(\eta(\cdot))\in C([0,T],\mathcal{L}(Y_{1}^{j+1},Y_{0}^{j+1})).
$$

Condition (iii). Similarly to \eqref{Flipcont} we have
\begin{eqnarray*}
\|F(\eta(\cdot),\cdot)\|_{ s+(j+1)b,p}\leq C_{14}\|\eta^{\frac{m-1}{m}}(\cdot)\|_{B^{s+jb+2\sigma-\frac{2\sigma}{q}}_{p,q}}\|f\|_{ s+(j+1)b,p}
\end{eqnarray*}
for certain $C_{14}>0$. By Proposition~\ref{embBCH} (iii), the set $\eta^{\frac{m-1}{m}}(t)$, $t\in[0,T]$, is bounded in $B^{s+jb+2\sigma-\frac{2\sigma}{q}}_{p,q}(\M)$, so that $F(\eta(\cdot),\cdot)\in L_{q}((0,T),Y_{0}^{j+1})$. 

We conclude that for each $\delta\in(0,T)$
$$
 w\in \bigcap _{\nu>0} L_q((\delta,T), H^{\nu+2\sigma}_p(\M))\cap H^1_q((\delta,T), H^\nu_p(\M)). 
$$
By the same argument as before, we can pass the above regularity to $u$ so that by \cite[(I.2.5.2)]{Ama95}, Proposition \ref{Prop: interpolation} and \eqref{inembmaxreg} we also have
\begin{equation}\label{usmooth1}
 u\in \bigcap _{\nu>0} H^1_q((\delta,T), H^\nu_p(\M))\cap C([\delta,T),H^\nu_p(\M)). 
\end{equation}
Thus, by Proposition~\ref{embBCH}
\begin{equation}\label{usmooth2}
 u^{m-1}\partial_{t}u\in \bigcap _{\nu>0} L_q((\delta,T), H^\nu_p(\M)), 
\end{equation} 
so that, by differentiating \eqref{FPMEQ} over time we find that $u\in H^2_q((\delta,T), H^\nu_p(\M))$ for all $\nu>0$. Then \eqref{regFPME2} follows by iteration. 
\end{proof}
 
Higher regularity for solutions of \eqref{FPMEQ} in $\mathbb{R}^{n}$ was recently proved in \cite{VPGR} by differentiating in time or by applying $(-\Delta)^{\sigma}$ to the equation and then using successively a H\"older continuity result, see \cite[Section 6]{VPGR}. In the proof of Theorem \ref{Thsmoothfpme} above, we follow a totally different method, based on abstract maximal regularity theory, which is of particular interest by itself. 
 
\subsection{Nonlocal Kirchhoff diffusion problem}\label{Section 6.3}

The following parabolic Kirchhoff type problem
\begin{equation}
\label{Kirchhoff eq}
\left\{\begin{aligned}
\partial_t u(t) -M(\| \nabla u\|_2^2)\Delta u(t) &=0, &&t\geq 0,\\
u(0)&=u_0, &&
\end{aligned}\right. 
\end{equation}
has been studied by many authors; see \cite{Gobbino} and the references therein.
Here 
\begin{equation}
\label{APS: M of Kir}
M\in C^{1-}([0,\infty) , [0,\infty)).
\end{equation}

We will consider a tensor-valued non-local version of \eqref{Kirchhoff eq}, namely,
\begin{equation}
\label{Kirchhoff eq - nonlocal}
\left\{\begin{aligned}
\partial_t u(t) + M(\| (-\Delta)^{\sigma/2} u\|_2^2)(-\Delta)^\sigma u(t) &=F(u), &&t\geq 0,\\
u(0)&=u_0 . &&
\end{aligned}\right. 
\end{equation}
This equation has been explored in \cite{Gobbino}. 
Here $(\M,g)$ is either an $n$-dimensio\-nal closed manifold or $\R^n$.
Let $X_1^s=H^{s+2\sigma}_p(\M,V)$ and $X_0^s=H^{s}_p(\M,V)$. 
Put 
$$
U_p^s=\{ v\in (X_0^s,X_1^s)_{1-1/p,p} \,:\, M(\| (-\Delta)^{\sigma/2} v\|_2^2)\neq 0\}, \quad s+\sigma>2\sigma/p,
$$
and assume that 
\begin{equation}
\label{APS: F of Kir}
F\in C^{1-}(U^s_p , X_0^s).
\end{equation}
Similar problems have been investigated in \cite{Pucci17, Xiang18} for the scalar case with $F$ independent of $u$ and $F(u)=|u|^{r-2}u$, where $1<r<\infty$. 
The theorem below generalizes the previous results \cite{Pucci17, Xiang18} on nonlocal Kirchhoff equations to the tensor-valued case and extends the admissible class of nonlinearities. 
In particular, our result applies to systems of nonlocal Kirchhoff equations.

By Theorem~\ref{main theorem}, Theorem~\ref{ClementLi} and \cite[Theorem~3.1]{RoSch18}, we immediately have the following result.
\begin{theorem}
Assume that \eqref{APS: M of Kir} and \eqref{APS: F of Kir} are satisfied. For any $p\in (2,\infty)$ and
$u_0\in U^0_p$, the equation \eqref{Kirchhoff eq - nonlocal} has a unique solution 
$$
u\in L_p((0,T), X_1^0) \cap H^1_p((0,T),X_0^0)
$$
such that $M(\|(-\Delta)^{\sigma/2} u(t)\|_2^2)\neq 0$ for all $t\in (0,T)$. Moreover, 
$$
u\in L_p((\varepsilon,T), X_1^s) \cap H^1_p((\varepsilon,T),X_0^s) 
$$
for all $s\geq 0$ and $\varepsilon\in (0,T)$.
\end{theorem}

\end{document}